\documentclass[12pt]{article}

\usepackage{epsfig}

\usepackage{amssymb, amsmath, amsfonts, amsthm, graphics, amscd}
\usepackage{authblk}
\usepackage{comment}
\usepackage{subfigure}

\usepackage{pgf}
\usepackage{pgfplots}
\usepackage{latexsym, eepic, epic}

\usepackage{mathtools}

\newcommand{\mathsym}[1]{{}}

\def\R{{\mathbf R}}

\newtheorem{theorem}{Theorem}[section]
\newtheorem{definition}{Definition}[theorem]
\newtheorem{proposition}{Proposition}[theorem]
\newtheorem{corollary}{Corollary}[theorem]
\newtheorem{lemma}[theorem]{Lemma}

\begin{document}

\title{
The Combinatorial Data Fusion Problem\\
in Conflicted-supervised Learning
}
\author[1]{R. W. R. Darling}
\author[2]{David G. Harris}
\author[1]{Dev R. Phulara}
\author[3]{John A. Proos}
\affil[1]{Mathematics Research Group, National Security Agency, USA}
\affil[2]{Dept. of Computer Science, University of Maryland, USA}
\affil[3]{Tutte Institute for Mathematics \& Computing, Canada}

\date{\today}
%\classification{\large UNCLASSIFIED}

\maketitle

% Abstract
\begin{abstract}
The \textit{best merge} problem in industrial data science
generates instances where disparate data sources place incompatible relational
structures on the same set $V$ of objects. 
Graph vertex labelling data may include (1) missing or erroneous labels,
(2) assertions that two vertices carry the same (unspecified) label,
and (3) denying some subset of vertices from carrying the same label.
\textit{Conflicted-supervised learning} applies to cases where
no labelling scheme satisfies (1), (2), and (3).
Our rigorous formulation starts from
a connected weighted graph $(V, E)$, and an independence system $\mathcal{S}$ on $V$,
characterized by its circuits, called \textit{forbidden sets}.
Global incompatibility is expressed by the fact $V \notin \mathcal{S}$. 
\textit{Combinatorial data fusion} seeks a subset $E_1 \subset E$ of maximum
edge weight so that no vertex component of the subgraph $(V, E_1)$ contains any
forbidden set. 
Multicut and multiway cut are special cases where all forbidden sets
have cardinality two. The general case exhibits unintuitive properties, shown in counterexamples.
The first in a series of papers concentrates on cases where $(V, E)$ is a tree, 
and presents an algorithm on general graphs, in which the combinatorial data fusion problem
is transferred to the Gomory-Hu tree, where it is solved using greedy set cover.
Experimental results are given. 
\end{abstract}

\begin{small}
\end{small}
% Table of Contents
%\tableofcontents

\section{Machine learning context}

\subsection{Vertex labelling problems}
Categorical machine learning tasks may be cast as vertex labelling problems
on large, edge-weighted graphs. Examples include
(1) mapping the pixels of an image to parts of an object depicted in the image, 
(2) mapping web sites to functional categories of web services,
(3) discriminating between fraudulent and authentic credit card transactions,
(4) allocating mobile handsets to specific communication channels.
Suppose $V_o$ denotes the objects to be labelled, and $V_c$ the set of labels.
Information available to us is of three kinds:
\begin{enumerate}
\item
\textbf{Partial labelling: }
Some objects in $V_o$ already carry labels.
Express this by an edge $\{v, \phi(v)\}$ between $v \in V_o$ and its label $\phi(v) \in V_c$,
weighted to reflect strength of association.

\item
\textbf{Label equality: }
For some pairs $u, v \in V_o$, a weighted edge $\{u, v\}$ expresses evidence that $u$ and $v$
should carry the same (as yet undeclared) label.

\item
\textbf{Prohibitions: } There are some subsets of $V_o$ called \textbf{forbidden sets}.
Not all vertices in a forbidden set may be assigned the same label. 
In data base contexts, for example, such prohibitions occur
 when there is evidence that two objects with similar or identical names are actually distinct.

\end{enumerate}

\subsection{Conflicted-supervised learning}
The three kinds of information mentioned above typically come from disparate sources, and
may partially contradict each other. In
Table \ref{t:ml}, we call this \textbf{conflicted-supervised learning}.
The data scientist is asked to perform a \textit{best merge}
of these data sources to arrive at a vertex labelling. 

In a semi-supervised learning graph problem, by contrast, few vertices
have been labelled, and the task is to extend this labelling to all vertices in a way
that respects the adjacency structure, as in \cite{zha}. Such label \textbf{propagation} in graphs 
can be achieved through random walks, belief propagation, semi-definite programming, etc.; 
see references in Darling \& Velednitsky \cite{dar}.

Combinatorial data fusion, as introduced here, is a rigorous approach to conflicted-supervised learning.
By deleting graph edges, break the graph
into components, each of which contains at most one of the label vertices in $V_c$.
If $v \in V_o$ is in a component with a unique $V_c$ vertex then $v$ receives that as its label.  For a component without a $V_c$ vertex a new distinct label is created and all vertices within that component receive the new label. Note that the label received by a vertex may not agree 
with the vertex's original label, and some label vertices may end up isolated.
The aim is to minimize the total weight of deleted edges of types 1. and 2.,
so to enforce all the prohibitions in 3.
Instead of label propagation, existence of forbidden sets leads to
label \textbf{dissociation}. The approach taken here builds on the machinery of graph cuts.

\begin{table}
\centering
\begin{tabular}{ |c|c| }
\hline
Problem Class & Label Amount \\
\hline
\textsc{Unsupervised Learning} & none \\
\textsc{Semi-supervised Learning} & partial \\
\textsc{Supervised Learning} & all \\
\textsc{conflicted-supervised Learning} & excess $+$ conflict \\
\hline
\end{tabular}
\caption{Combinatorial data fusion is a rigorous approach to conflicted-supervised learning. } \label{t:ml}
\end{table}

\subsection{Example: product categorization as combinatorial data fusion}
For the sake of concreteness, here is an over-simplified example, which we abstract in
Section \ref{s:anomaly}. You are the manager
of an online market place in which various vendors offer products for sale, whose union is a set $V_o$.
Vendors are expected to label each item in $V_o$ with one of a set $V_c$ of exclusive
categories, but some vendors fail to do so, or label incorrectly, perhaps because the categories
are difficult to understand. Set $V:=V_o \cup V_c$ as a vertex set, with two types of weighted edges:
\begin{enumerate}
\item
If vendor $k$ assigns product $v \in V_o$ to category $c :=\phi_k(v)\in V_c$, place weight
$w_k > 0$ on edge $\{v, c\}$. Less reliable vendors receive less weight.
\item
If transaction data indicate that $m$ buyers have bought both $v, v' \in V_o$,
assign weight $w_m' > 0$ to edge $\{v, v'\}$, where $m \mapsto w_m'$ is increasing.
The idea is that common buyers suggest two items belong to the same category.
\end{enumerate}
In exclusive categorization, the forbidden sets are $\{ \{c, c'\}, c \neq c' \in V_c\}$, and
combinatorial data fusion becomes a multiway cut problem, discussed in Section \ref{s:mmwcut}.
When categories are difficult to understand, a forbidden set might take the form $\{c, c', c''\}$,
meaning that the three categories are not identical, but possibly two out of three
may coincide. In summary, the goal is best merge of vendor labelling data with sales transaction
data. The output of the algorithm will be an assignment of each product to a category, possibly different
to the one assigned by the vendor.

\subsection{Forthcoming works}
The present work covers fundamental definitions and functorial properties, and
solution methods when the underlying graph is a tree, or is approximated by a tree.
Future works will cover linear programming relaxations,
and cases where forbidden sets are not given explicitly, but are returned by an
oracle which decides whether a system of linear inequalities has a feasible point.

\section{Problem statement}

\subsection{Circuits of an independence system}

Given a set $V$, and a collection $\mathcal{S}$ of subsets of $V$, we call 
$(V, \mathcal{S})$ an \textbf{independence system} if
$U \in \mathcal{S}$ and $W \subset U$ implies $W \in \mathcal{S}$, and $\emptyset \in S$.
Elements of $\mathcal{S}$ are called \textbf{independent sets}. 
A subset $W \subset V$ which is not in $\mathcal{S}$ is called a \textbf{dependent set}.
Minimal dependent sets are called \textbf{circuits}. 
The collection of circuits of an independence system $(V, \mathcal{S})$ is denoted $\mathbf{C}(\mathcal{S})$.
% Maximal independent sets are called \textbf{bases}.
The next assertion is part of Korte and Vygen \cite[Theorem 13.12]{kor}.

\begin{lemma} \label{l:circuits}
Let $\mathcal{F}$ be a collection of non-empty subsets of $V$, such that $F_1, F_2 \in \mathcal{F}$
and $F_1 \subseteq F_2$ implies $F_1 = F_2$. 
Let $\mathcal{S}$ consist of those $U \subset V$ which contain no element of $\mathcal{F}$
as a subset. Then $(V, \mathcal{S})$ is an independence system, whose set of circuits
$\mathbf{C}(\mathcal{S}) = \mathcal{F}$.
\end{lemma}
In practical applications, the most parsimonious specification of an independence system may be its set of circuits.
Circuits are discussed further in Schrijver \cite[vol. B, Section 39.6]{sch}.

\subsection{Two combinatorial structures on the same vertex set} \label{s:2cs}

Let $V$ be a finite vertex set on which two separate combinatorial structures are defined:
\begin{enumerate}

\item
An edge-weighted graph $G = (V,E_0, w)$, with weights $w:E_0\rightarrow  (0, \infty)$. 
With no loss of generality, assume $G$ is connected\footnote{
If $G$ is not connected, then each graph component may be treated as a separate problem.}.

\item
An independence system $\mathcal{S}$, which includes all singletons\footnote{
This implies every forbidden set has cardinality at least two.
} and all edges in $E_0$ (i.e. $\{ \{v\} | v  \in V\} \cup E_0 \subset S$).
The circuits $\mathbf{C}(\mathcal{S}) =: \mathcal{F}$
are called \textbf{forbidden sets}. Call $H:=(V, \mathcal{F})$ the \textbf{forbidden hypergraph},
and call $A:=(V, \mathcal{S})$ the \textbf{allowable hypergraph}.

\end{enumerate}

\begin{figure}
\caption{ \textbf{EXAMPLE: }\textit{A connected graph with sixteen edges and eleven vertices 
was generated by a randomized method. 
Independent uniform$(0,1)$ random variables were used to assign the edge weights. 
The combinatorial data fusion problem, with  
forbidden sets $\{v_1, v_2, v_3\}$ and $\{v_5, v_8\}$, was solved by exact search, giving
a subgraph whose edges are shown with solid lines; omitted edges 
(total weight $2.648$) are shown as dashed lines. 
Components of this subgraph were  $\{v_1, v_7, v_8, v_{10}, v_{11}\}$, and  $\{v_2, v_3, v_4, v_5, v_6, v_9\}$,
respectively.
The edge $\{v_1, v_3\}$ has maximum weight $0.944421$ out of all sixteen edges, 
but is absent from the optimum subgraph.
This example demonstrates that Best-In Greedy Algorithm,
which would have included $\{v_1, v_3\}$ as its first choice, need not deliver an optimum.
}
} \label{f:greedynonopt}

\begin{center}
\scalebox{0.375}{\includegraphics{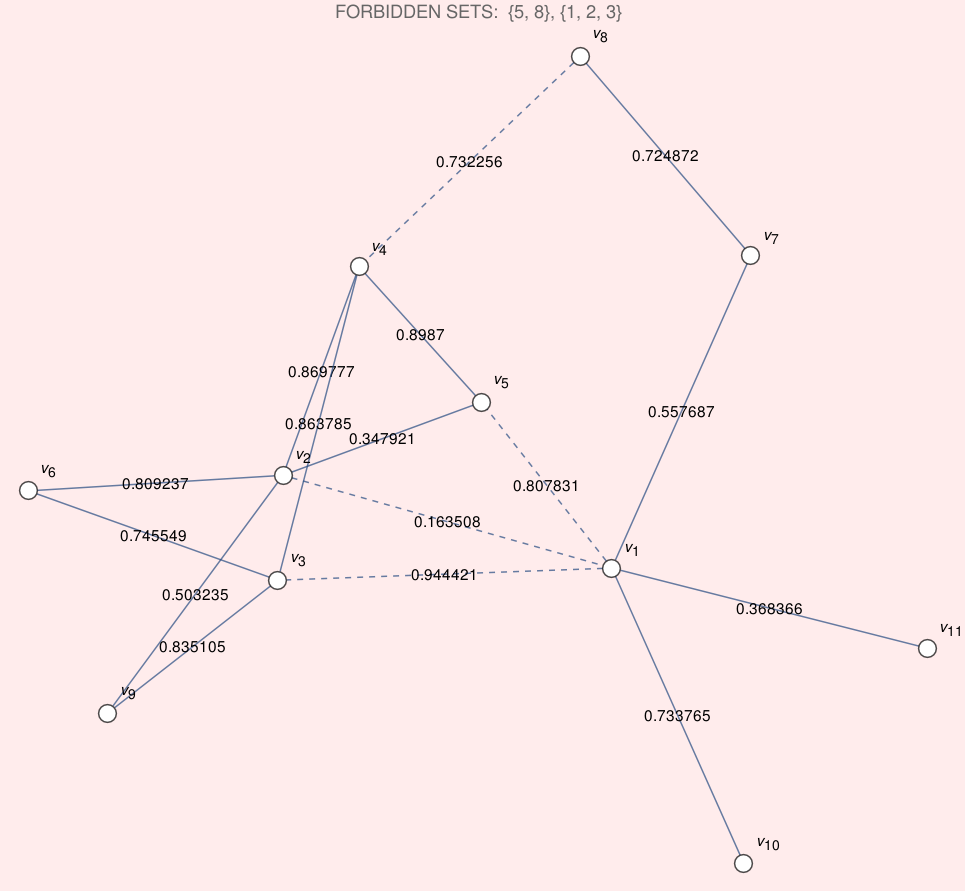}} 
\end{center}
\end{figure}

\subsection{Combinatorial data fusion problem}
What we call \textbf{data fusion} occurs between the forbidden hypergraph 
$H=(V, \mathcal{F})$, and the weighted graph $G = (V,E_0, w)$. More specifically, the data fusion problem is looking for an optimal, with
respect to $w$, subgraph of $G$ which avoids all the structures that the forbidden hypergraph disallows. To make this precise, three
 combinatorial formulations of the same problem are given below.

\subsubsection{Maximum subgraph whose components include no forbidden sets} \label{s:msg}
A subgraph formulation of the problem is as follows.

 A subset $E_1 \subseteq E_0$ induces a partition of $V$ into disjoint graph components 
$V_1 \cup \cdots \cup V_d$, for some $d \geq 1$.
Take $\mathcal{E}$ to be those subsets $E_1 \subset E_0$, called \textbf{feasible}
subgraphs, such that none of the 
induced graph components of $(V, E_1)$ contains a forbidden set. 
Then $\mathcal{E}$ is an independence system, but typically not a matroid, as Figure \ref{f:extraconstraintcounterex} shows.
We seek $E_1 \in \mathcal{E}$ to maximize the sum of weights:
\begin{equation} \label{e:mp4is}
\sum_{e\in E_1}w(e).
\end{equation}
This is an instance of what Korte and Vygen \cite{kor} call
the \texttt{maximization problem for independence systems}.

The subgraph formulation leads naturally to a
\textbf{Best-In Greedy Algorithm}: starting with the empty graph on $V$,
insert edges of $E_0$ in decreasing weight order, omitting any edge whose insertion
would create a component containing a forbidden set.
The failure of this algorithm to deliver an optimum solution
is shown in Figure \ref{f:greedynonopt}. For the example in Figure \ref{f:greedynonopt} 
the optimal $E_1$ is
$E \setminus \{\{v_1,v_2\}, \{v_1,v_3\},\{v_1,v_5\},\{v_4,v_8\}\}$.

\subsubsection{Optimum coloring of the forbidden hypergraph} \label{s:fhg}
A coloring formulation of the problem is as follows.

%The number $d$ of colors will not be specified in advance. 
A \textbf{hypergraph coloring} of $(V, \mathcal{F})$
means a mapping $\chi: V \to \mathbb{N}$ such that no hyperedge in $\mathcal{F}$ is monochromatic;
in other words, there is no integer $j$ or forbidden set $F$ for which $F \subseteq \chi^{-1}(j)$.

The optimization problem is to find a hypergraph coloring $\chi: V \to \mathbb{N}$
for which the total weight under $w:E_0\rightarrow \mathbb (0, \infty)$ of edges with differently
colored endpoints is minimum. The objective function is
\begin{equation} \label{e:hgcof}
\min_{\chi} \sum_{\substack{
 \{u, v\} \in E_0 \\
 \chi(u) \neq \chi(v)}
} w(\{u, v\}).
\end{equation}
As the sum of weights of all edges in $E_0$ is fixed, an equivalent objective function is  
\begin{equation} \label{e:hgcof_max}
\max_{\chi} \sum_{\substack{
 \{u, v\} \in E_0 \\
 \chi(u) = \chi(v)}
} w(\{u, v\}).
\end{equation}

For the example in Figure \ref{f:greedynonopt} the optimum coloring is
\[
\chi(v_i) = 
\begin{cases}
1 & \textrm{if } i \in \{1,7,8,10,11\}, \\
2 & \textrm{if } i \in \{2,3,4,5,6,9\}.
\end{cases}
\]
Comparing this optimum coloring to the optimum subgraph given in Section \ref{s:msg} 
we see that the optimum coloring
 assigns a distinct color to each connected component of the optimum subgraph and that the 
edges with different colored endpoints are exactly the edges removed from $E$ to form the optimum subgraph.

The coloring formulation leads naturally to a
 \textbf{Greedy Coloring Algorithm}, extending the Greedy Graph Coloring in
 Korte and Vygen \cite[Section 16.2]{kor}. Order\footnote{
Possibly at random. Or possibly depending on total weight of incident edges, 
or on vertex degree in the forbidden hypergraph.
} the vertices as $V:= \{v_1, v_2, \ldots, v_n\}$.
For each vertex $v_i$, let $\mathcal{F}_i \subset \mathcal{F}$ consist of those forbidden sets $F$
for which $v_i$ is the highest numbered vertex; in other words,
$F \ni v_i$ and
\[
F \cap \{v_{i+1}, \ldots, v_n\} = \emptyset.
\]
Start by assigning $\chi(v_1) = 1$. 
For $1 < i \leq n$, let $C = \{1, \dots, k \}$ be the current set of colors used and let $C' \subset C$ be those colors $c$ for which there is no $F \in \mathcal{F}_i$ with all the vertices in $F\setminus \{v_i\}$ colored $c$. There are two cases: (1) If $C' = \emptyset$ then assign $\chi(v_i)$ the new color $k+1$; (2) Otherwise assign $\chi(v_i)$ a $k \in C'$ which maximizes the total weight of edges with endpoints of the same color. 
This procedure guarantees that no forbidden set will be monochromatic. 
This coloring is parsimonious rather than greedy, in the sense that introduction of new
colors is delayed until it is inevitable. When implemented on
the example in Figure \ref{f:greedynonopt}, using the given vertex order, the result is
\[
\chi(v_i) = 
\begin{cases}
1 & \textrm{if } i \in \{1,2,4, 5, 6, 7,10,11\}, \\
2 & \textrm{if } i \in \{3,8,9\}.
\end{cases}
\]

\subsubsection{Maximum weight matching of the allowable hypergraph} \label{s:mwm}
A \textbf{matching} of a hypergraph is a collection of
pairwise disjoint hyperedges. A matching is called a \textbf{perfect
matching} if every vertex of the hypergraph is contained in a matching edge. A
matching formulation of the problem is as follows.

Use the edge-weighted graph $G = (V,E_0, w)$ to turn the allowable hypergraph
into a \textbf{weighted hypergraph} $A_w:=(V,\mathcal{S},w)$,
by taking the weight $w(U)$  of $U \in \mathcal{S}$
 to be the total edge weight of the subgraph of $G$
induced on the vertex set $U \subset V$. 
Let $\mathcal{M}$ denote the set of matchings of the allowable hypergraph.
We seek a matching $M \in \mathcal{M}$ which maximizes the sum of the weights
of the matching hyperedges:
\begin{equation} \label{e:hgmwm}
w(M) := \sum_{U \in M} w(U).
\end{equation}
Since $\{v\} \in \mathcal{S}$ for all $v \in V$ by assumption, and all
such singleton hyperedges have weight zero, any matching of the allowable
hypergraph can be extended to a perfect matching of equal weight. This allows us
to restrict our maximum weight matching problem to perfect matchings. 

Let's consider the example in Figure \ref{f:greedynonopt} again. 
The optimum perfect matching for this problem is 
$\{\{v_1,v_7,v_8,v_{10},v_{11}\},\{v_2,v_3,v_4,v_5,v_6,v_9\}\}$.
Comparing this matching to the optimum coloring given in Section \ref{s:fhg}
 we see that the optimum coloring
 mono-colors each of the hyperedges of the optimum perfect matching with a distinct color. 

The matching formulation leads naturally to a
 \textbf{Greedy Matching Algorithm}. Start with the trivial weight
zero perfect matching $M_0$ consisting of all the singleton hyperedges.
Given a (perfect) matching $M_1$ of $A_w$ and $U_1,U_2
  \in M_1$ such that $U_1 \cup U_2 \in \mathcal{S}$, the collection of hyperedges
  $M_2 := (M_1 \setminus \{U_1,U_2\}) \cup \{U_1 \cup U_2\}$ is a (perfect)
  matching of $A_w$ and $w(M_2) \geq w(M_1)$. Call the latter a merge operation, with a weight gain of 
$w(M_2) - w(M_1)$.
At each step of Greedy Matching, perform a merge operation with maximum weight gain\footnote{
This is not the same as the Best-In Greedy Algorithm
of Section \ref{s:msg} because the weight of all edges between $U_1$ and $U_2$ is considered
when computing the weight gain.
}. Stop when no further merges are possible.

\subsection{Interchangeability of three formulations}

The \textbf{combinatorial data fusion problem} $(V, E_0, w, \mathcal{F})$
refers to any of the three optimization problems above.
Indeed:

\begin{lemma} \label{l:equiv}
The maximum subgraph (Section \ref{s:msg}), optimum coloring (Section \ref{s:fhg}) and maximum matching (Section \ref{s:mwm})
formulations of combinatorial data fusion are interchangeable in the following sense:
an optimum solution to any one of the formulations yields, in linear time, feasible solutions to the other two problems, with objective
function values equal to the value of the optimum solution to the initial formulation (using the maximization version of the objective function
for the optimum coloring formulation).
\end{lemma}

\textbf{Remark: } An optimum may be called an \textbf{optimum hypergraph coloring}
in the context (\ref{e:hgcof}), a \textbf{maximum subgraph} in the context (\ref{e:mp4is}),
or a \textbf{maximum hypergraph matching} in the context (\ref{e:hgmwm}).

%JAP: I have removed this remark because the example given was not correct. The Maximum Weight Independent Set (MWIS) is a generalization of CDF not a
% special case. Thus we cannot conclude that there is no approximation algorithm. We showed above that CDF is an instance of MWIS, thus if we could
% show that every MWIS problem can be formulated as a CDF then we would have that that the problems are equivalent. 
% If we can fix the example then we can add the remark back. 

% \textbf{Remark 2: } While the minimization formulation (Section \ref{s:fhg}) and maximization
%formulation (Section \ref{s:msg}) are equivalent in terms of exact solution,
%they need not be equivalent in terms of approximating the optimal solution. 
%A classic example is presented in detail in the sequel \cite{cdf2},
%
%\begin{enumerate}
%\item
%No $n^{1 - \epsilon}$ approximation exists\footnote{
%See Trevisan \cite{tre} for a survey of inapproximability results.}
%for $\epsilon<1$, for Maximum Weight Independent Set
%on $n$ vertices, which is a special case of
%the maximization problem (Section \ref{s:msg}).
%\item
%The complement of a Maximum Weight Independent Set corresponds to the edges
%deleted in a Multicut (Section \ref{s:multicut}) problem on a tree,
%for which there is a factor 2 approximation algorithm; see \cite{shm} Exercise 7.2.
%\end{enumerate}

\begin{proof}
Let $G=(V,E_0,w)$,$H=(V,F)$ be an instance of a combinatorial data fusion problem.

Case 1: Suppose that $E_1$ is the edge set of an optimum solution to the subgraph formulation of the problem.
The objective function value of $E_1$ is $\sum_{e \in E_1} w(e)$.

Let $\chi$ be a coloring of $H$ that monocolors the vertices of each connected component of $(V,E_1)$ a distinct color. As $E_1$ was a
feasible solution to the maximum subgraph formulation of the problem, no forbidden set of $H$ is contained within a single component of $(V,E_1)$. Thus $\chi$
is a valid coloring of $H$. Additionally, the objective function value for the maximization version of the optimum coloring formulation for $\chi$ is
the sum of the $E_0$ edges whose ends are in the same component of $(V,E_1)$, which by the optimality of $E_1$ is $\sum_{e \in E_1} w(e)$.

Let $M$ be the partition of $V$ consisting of the vertex sets of the connected components of $(V,E_1)$. As no connected component of $(V,E_1)$
contains a forbidden set, $M$ is a perfect matching of the weighted allowable hypergraph $A_w$ (see Section \ref{s:mwm}). By the
definition of $A_w$, the weight of each hyperedge $m$ in $M$ equals the sum of the weights of edges of $E_0$ with both
ends in $m$.  By the optimality of $E_1$, the weight of each hyperedge $m$ in $M$ equals the sum of the weights of edges
of $E_1$ with both ends in $m$ and each edge $e \in E_1$ counts towards the weight of exactly one $m \in M$. Thus the objective function value of
$M$ is $\sum_{e \in E_1} w(e)$. 

Case 2: Suppose that $\chi$ is an optimum coloring of the forbidden hypergraph. The value of the maximization version of the objective
function for $\chi$ is 
\begin{equation*} \sum_{\substack{
 \{u, v\} \in E_0 \\
 \chi(u) = \chi(v)}
} w(\{u, v\}),
\end{equation*}
which we will denote by $Z$.

Let $E_1 = \{e \in E |  \chi(u) = \chi(v)\}$. By construction the objective function value for $E_1$ in the maximum subgraph formulation is $Z$ as
required.
Since all edges of $(V,E_1)$ join vertices with equal colors, every connected component of $(V,E_1)$ is monochromatic. Since $\chi$ is a feasible
coloring of the forbidden hypergraph it follows that no forbidden set is contained in a connected component of $(V,E_1)$.

Let $M$ be the partition of $V$ consisting of the preimages of the colors of $\chi$. As
$\chi$ is a valid coloring of the forbidden hypergraph, $M$ is a perfect matching of the
weighted allowable hypergraph $A_w$. By the definition of $A_w$, the weight of each hyperedge $m$ in $M$ equals the sum
of the weights of edges of $E_0$ with both ends in $m$.  Thus the objective function value of $M$ is $Z$.

Case 3: Suppose that $M$ is an optimum perfect matching of the weighted allowable hypergraph $A_w$. The objective function value for $M$ is 
\begin{equation*} 
\sum_{m \in M} w(m) = 
\sum_{m \in M}  \sum_{ \{u, v\} \in E_0 \cap m} w(\{u, v\}) =
\sum_{\substack{
 \{u, v\} \in E_0 \\
 M(u) = M(v)}
} w(\{u, v\}),
\end{equation*}
where $M(u)$ is the $m \in M$ for which $u \in m$. Denote this objective function value by $Z$.

Let $E_1 = \{ \{u,v\} \in E_0 | M(u)=M(v)\}$. By construction $\sum_{e \in E_1} w(e) = Z$. As all edges of $E_1$ join vertices within an $m \in M$,
if $C$ is a connected component of $(V,E_1)$ then there exists an $m \in M$ such that all vertices of $C$ are in $m$. As $M$ is a perfect matching of
$A_w$ it follows that no forbidden set can be contained in a connected component of $(V,E_1)$.

Form a hypergraph coloring $\chi$ which monocolors each $m \in M$ a distinct color. Since each $m \in M$ is a hyperedge of the allowable hypergraph,
no forbidden set is monocolored by $\chi$. The objective function value of $\chi$ is 
\[
\sum_{\substack{
 \{u, v\} \in E_0 \\
 \chi(u) = \chi(v)}
} w(\{u, v\}) = 
\sum_{\substack{
 \{u, v\} \in E_0 \\
 M(u) = M(v)}
} w(\{u, v\}) 
=Z
\]
 
%Old proof which didn't proof the new version of the lemma
%Under the versions of the problem spelled out in Sections \ref{s:msg}, \ref{s:fhg},
%and \ref{s:mwm}, respectively, 1., 2., and 3. below describe the same partition $V_0, V_1, \ldots, V_{d-1}$
%of $V$.
%\begin{enumerate}
%\item
%The components of the graph $(V, E_1)$, where
%$E_1 \subset E_0$ is a maximum subgraph.

%\item
%The collection, ranging over colors $i$ in an optimum hypergraph coloring $\chi:V \to \mathbb{N}$,
%of the connected components\footnote{ This $d$ need not be a minimal set 
%of colors: for example, it may be that for $i \neq j$, no vertex in $\chi^{-1}(j)$
%is adjacent in $E_0$ to any vertex in $\chi^{-1}(i)$, and so colors $i$ and $j$ could be
%merged to give a coloring $\chi:V \to \{0, 1, \ldots, d-2\}$ which achieves the same optimum.}
%of the subgraph of $(V, E_0)$ induced on the pre-image $\chi^{-1}(i) \subset V$ of  color $i$.

%\item
%The collection of hyperedges\footnote{
%Possibly two or more of these hyperedges could be merged, to give another
%matching of the same weight.
%} in the allowable hypergraph which
%form a maximum hypergraph matching $\mathcal{M}$.

%\end{enumerate}
%In every case, every set in the partition is an allowable set. In cases 1. and 3.,
%the objective function to be maximized is the total weight of edges
%in $E_0$ both of whose endpoints are in the same set of the partition.
%In case 2., the objective function to be minimized is the total weight of edges in the
%complementary set. 
Therefore the result holds in all three cases.
\end{proof}

\subsection{Special case: multi-multiway cut} \label{s:mmwcut}

Let $G = (V,E_0, w)$ be a weighted graph and let $S_1, S_2, \dots,
S_k$ be $k$ sets of vertices. Avidor \& Langberg \cite{avi} define the 
\texttt{multi-multiway cut problem} (MMC) as finding a subset $E_2 \subset E_0$ of
minimum total weight such that the $S_i$ are completely disconnected in $(V, E_0 \setminus E_2)$ 
(i.e. if $u,v \in S_i$ for some $i$ then $u$ and $v$ are
disconnected in $(V, E_0 \setminus E_2)$ ). The \texttt{multiway cut problem} is the 
sub-case where $k=1$, meaning that there is only one set of vertices that have to
be pairwise disconnected; see Chapter 8 of Williamson \& Shmoys
\cite{shm} and Chapters 4 and 19 of Vazirani \cite{vaz}. 
Dahlhaus et al \cite{dah} proved this problem to be NP-Hard.
Recently Velednitsky \& Hochbaum \cite{vel} showed that
an algorithm which gives a 2-approximation actually gives an optimum
when the graph is 2-stable.
The more familiar \texttt{multicut problem},
discussed in section \ref{s:multicut}, is the sub-case where $|S_i| = 2$ for $1
\leq i \leq k$.  This is NP-hard except when $k = 2$, and is
analyzed at length by Williamson \& Shmoys \cite{shm} and Vazirani \cite{vaz}.

Contrast the multi-multiway cut problem with the combinatorial data fusion problem
on the same $G$ where $S_1, S_2, \dots,S_k$ are the forbidden sets. For the latter,
removal of edges needs only to ensure that the $S_i$ are not connected, whereas
in the former each $S_i$ must be completely disconnected.
The multi-multiway cut problem is a special case of the combinatorial data fusion problem where the forbidden sets are all pairs of vertices contained in the $S_i$ (i.e. $\mathcal(F) = \{ (u,v) | u,v \in S_i, 1 \leq i \leq k \}$).

%\begin{definition}
%A combinatorial data fusion problem $(V, E_0, w, \mathcal{F})$ is called a
%\textbf{MMC data fusion} problem if the allowable subsets $\mathcal{S}$
% have the property that if $ F' \subset F \notin \mathcal{S}$ and
%$|F'| \geq 2$ then $F' \notin \mathcal{S}$.
%\end{definition}

%For this subclass of problems we know
%that if $F \notin \mathcal{S}$ then a feasible set of edges $E_1 = E_0 \setminus E_2$,
%in the sense of Section \ref{s:msg}, must completely disconnect $F$.
%Thus this subclass of problems are in fact multi-multiway cut
%problems. 

%Additionally, given a multi-multiway cut problem, setting $\mathcal{F}
%= \bigcup_{1 \leq i \leq k} \{F \subset S_i : |F| \geq 2\}$ or $\mathcal{F} = \{\{u,v\}
%| u,v \in S_i, 1 \leq i \leq k, u \neq v\}$ forms equivalent
%MMC data fusion problems (same feasible solutions $E_1$ with the same weights).
%The second formulation of $\mathcal{F}$, relying on the fact
%that when $F_1,F_2 \in \mathcal{F}$ and $F_1 \subset F_2$ then $F_2$ can be 
%removed from $\mathcal{F}$ without altering the feasible $E_1$ or their weights,
%to give a quadratic reduction of multi-multiway cuts to combinatorial
%data fusion.

\subsubsection{Multi-multiway cut where the $S_i$ partition $V$}
Of interest is the subclass of the MMC data fusion problems for which the $S_i$
partition the vertex set. For this subclass there exists a surjective mapping
 from the vertex set to $k$ types, and the forbidden sets
$\mathcal{F}$ can be reduced to simply the edges $\{v, v'\}$ whose endpoints
have the same type, i.e. $\phi(v) = \phi(v')$. Thus the 
forbidden hypergraph is a graph with $k$ components, each of which is a complete
graph on the vertices of a given type.  Here we seek to cut edges of minimum total weight so no component
contains more than one vertex of any type.

\textbf{Human resources example: } A collection of ships is awaiting departure at a seaport.
Each ship has a captain, but lacks other officers. The rules say that no ship may contain
more than one exemplar of each of the officer types: navigator, cook, surgeon, carpenter, and purser.
Officers of these types are available at the seaport. The conviviality score between
two officers (of different types) is the number of previous voyages they have taken together.
The task is to hire officers for the ships, so as to maximize the total of
conviviality scores among officers on the same ship. Here vertices are officers,
type is officer type, and the coloring $\chi$ assigns each officer to a ship.

%Call a subset $U \subset V$ \textbf{duplicate-free} if all the vertices in $U$
%are of different types. Any collection $U_1, U_2, \ldots \subset V$ of
% disjoint, duplicate-free vertex sets is called a \textbf{hypergraph matching}
% for $(V, \mathcal{F})$.
%Let $\mathcal{M}$ denote the sets of hypergraph matchings for $(V,
% \mathcal{F})$.
%Define the \textbf{weight of the matching} $U_1, U_2, \ldots \subset V$ to be
%\begin{equation} \label{e:maxmatch}
%\sum_i \sum_{\substack{
% u, v \in U_i \\ \{u, v\} \in E_0}
%} w(\{u, v\}).
%\end{equation}
%Hence a maximum weight matching is exactly a coloring $\chi: V \to \mathbb{N}$
% which minimizes (\ref{e:hgcof}), where $u \in U_i$ is the same as $\chi(u) = i$.

%In summary, when the forbidden hypergraph is a graph whose components are
% complete graphs, optimum vertex coloring 
%reduces to maximum weight hypergraph matching with respect to $w:E_0\rightarrow
% \mathbb (0, \infty)$.

\subsection{Application where forbidden set cardinality exceeds two}\label{s:multicut}
The combinatorial data fusion problem extends multicut in that
forbidden sets of cardinality larger than two are allowed.
Here is an example showing the need for this greater generality.

Bishop and Darling \cite{bis} consider sporadic observation of a moving target.
Take an {\em observation set} 
$\mathcal{O}$, which consists of a collection of pairs 
\begin{equation}\label{e:reports} 
\mathcal{O}=\left\{(t_1, K_1), (t_2, K_2), \ldots, (t_{n}, K_{n})\right\} 
\end{equation} 
where $0 \leq t_1 <  t_2 < \cdots < t_{n} \leq T$ are the {\em observation times}, and 
each $K_i$ is a convex polygon $\R^d$ under the Manhattan metric.
In that case, membership in $K_i$ can be checked using a finite system of linear inequalities,
as can the non-emptiness of the set
\begin{equation} \label{e:convex}
\Lambda_B(\mathcal{O})= \left\{(\mathbf{y}_1, 
\mathbf{y}_2, \ldots,\mathbf{y}_n) \in K_1
\times K_2 \times \cdots \times K_n:
\max_{1\leq j\leq n-1 }
\frac{\|\mathbf{y}_{j+1} -\mathbf{y}_j\|}{t_{j+1} - t_j} \leq B \right\}
\end{equation}
where $B$ is a positive speed bound. To say that $\Lambda_B(\mathcal{O})$ is 
empty is equivalent to saying that a traveler with maximum speed $B$ could not
have visited zone $K_i$ at time $t_i$ for every $i$, in which case we call the
observation set \textbf{incompatible} with $B$; otherwise it is \textbf{compatible}.

Evidently the collection of subsets $X \subset \mathcal{O}$ such that $X$ is compatible with $B$
is an independence system, and its circuits can be taken as the forbidden sets $\mathcal{F}$ in
a combinatorial data fusion problem. 
The enumeration of such circuits is the
problem addressed by Boros et al \cite{bor}.
For example, a forbidden set of size three
consists of three observations $(t_1, K_1), (t_2, K_2), (t_{3}, K_{3})$ which are
incompatible with $B$, while every subset of size two is compatible. 
See Figure \ref{f:3incompatible} for a planar example under the Euclidean metric.

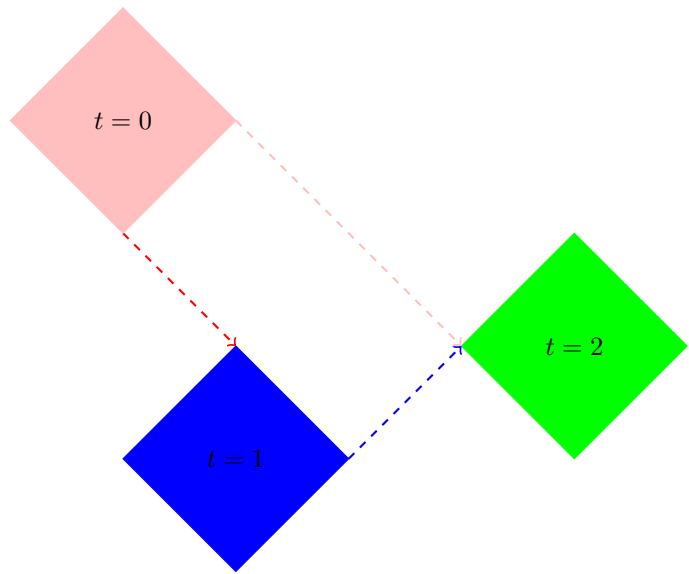
\begin{figure}
\caption{ \textbf{CIRCUIT OF SIZE THREE: }
\textit{
A traveler is observed in the left diamond $K_0$ at time $t=0$, the
lower diamond $K_1$ at time $t=1$, and the right diamond $K_2$ at time $t=2$. The
 two shorter dashed arrows each indicate the unique $\ell_1$ or $\ell_2$
 paths between respective diamonds that can be traveled in one time unit. 
The long dashed arrow is the only possible path from $K_0$
to $K_2$ in two time units.
All the pairs $\{(i, K_i), (j, K_j)\}$, for $0 \leq i < j \leq 2$,
 are compatible with this speed bound, but the triple 
$\{(0, K_0), (1, K_1), (2, K_2) \}$ is incompatible.
} 
}\label{f:3incompatible}
\begin{center}

\begin{tikzpicture}[scale=1.5]
\draw [color=blue, fill] (2.5,-1) -- (1.5,-2) -- (2.5,-3) -- (3.5,-2) -- (2.5,-1);
\node at (2.5,-2) {\footnotesize{$t = 1$}};
\draw [color = red, thick, dashed, ->] (1.5,0) -- (2.5,-1);
\draw [color=green, fill] (4.5,-1) -- (5.5,0) -- (6.5,-1) -- (5.5,-2) -- (4.5,-1);
\node at (5.5,-1) {\footnotesize{$t = 2$}};
\draw [color = blue, thick, dashed, ->] (3.5,-2) -- (4.5,-1);
\draw [color=pink, fill] (1.5,0) -- (2.5,1) -- (1.5,2) -- (0.5,1) -- (1.5,0);
\node at (1.5,1) {\footnotesize{$t = 0$}};
\draw [color = pink, thick, dashed, ->] (2.5,1) -- (4.5,-1);
\end{tikzpicture}

\end{center}
\end{figure}

At least two types of combinatorial data fusion problem appear in this context.
\subsubsection{Maximum set of compatible observations}
Suppose we are sure that all observations refer to the same traveler, but
 some observations must be wrong, because $\mathcal{O}$ itself is not compatible with $B$.
Construct an edge-weighted star
$(V, E, w)$, with the $\{(t_i, K_i), 1 \leq i \leq n\}$ as leaf vertices, and an auxiliary
vertex $v_o$ of degree $n$. The weight $w_i$ on the edge $\{v_o,  (t_i, K_i)\}$
measures trust in the reliabity of the observation, in some additive way.
The forbidden hypergraph $(V, \mathcal{F})$ consists of incompatible circuits
mentioned above. The combinatorial data fusion problem consists of finding
a set of compatible observations, namely those in the component of $v_o$,
of maximum total weight.

\subsubsection{Distinguishing travelers}
Tag each $(t_i, K_i)$ with a label $a_i \in V'$, where $V'$ is a set of
name strings. We believe $(t_i, K_i, a_i)$ is correct, but we do not know whether the name strings
$V'$ all refer to the same traveler, or to several different travelers. Create a connected
edge-weighted graph $(V', E, w)$, where the weight $w({u,v})$ on a pair of name strings $u, v$
represents the strength of external evidence that $u$ and $v$ refer to the same person,
such as similarity between the two strings. The forbidden sets $\mathcal{F}$
are the circuits in the independence system, whose elements are those $U \subset V'$
such that $\{(t_i, K_i), a_i \in U\}$ forms an observation set compatible with $B$.
In an optimum hypergraph coloring $\chi:V' \to \mathbb{N}$
for the combinatorial data fusion problem $(V', E, w, \mathcal{F})$, 
name strings with the same color
refer to observations which could refer to the same traveler. 
The number of distinct colors is the number of distinct travelers.
As we shall see in Section \ref{s:mincolnotopt}, 
this optimum need not correspond to a minimum number of distinct
travelers compatible with the observations. In such a context $V'$ could be small, while
$n$ could be large; hence determining the forbidden sets may be much more work than
solving the combinatorial data fusion problem itself. 
%In a sequel \cite{cdf3} we will consider forbidden set oracles.

%%%%%%%%%%%%%%%%%%%%%%%%%%%%%%%%%%%%%%%%%%%%%%%%%%%%%%%%%%%%%%%%%%%%%%%%%%%%%%%%%%%%%%%%%%%%
\section{Category theory perspective on combinatorial data fusion }

\subsection{Category of partitioned sets}

Carlsson and M{\'e}moli \cite{car} propose that a clustering algorithm should be a functor 
from a suitable category of finite metric spaces
into a suitable category of partitioned sets\footnote{
Carlsson and M{\'e}moli call this the category of outputs of standard clustering schemes.}.
Their categorical definitions are designed to achieve the following algorithmic property.
Suppose $\phi:X \to Y$ is an injective, non-expansive map between finite metric spaces.
Such a map could arise, for example, when a data set $X$ is augmented by additional data points.
Suppose points $x, x' \in X$ are assigned to the same cluster when a specific clustering algorithm
is applied to $X$; then $\phi(x), \phi(x') \in Y$ should be in the same cluster when the same algorithm is applied to $Y$.

Use the term \textbf{partitioned set} to denote a pair $(X, \mathcal{P})$, 
where $X$ is a set, and $\mathcal{P}$ is a partition of $X$. View $\mathcal{P}$ as a collection
of subsets $\{X_\alpha \}_{\alpha \in A}$ of $X$ which we call the
\textbf{blocks} of the partition.

Define the \textbf{category of partitioned sets} as follows.
Objects are partitioned sets $(X, \mathcal{P})$ with $X$ finite.
The morphisms  from $(X, \mathcal{P})$ to $(Y, \mathcal{Q})$ consist of set mappings $\phi:X \to Y$ 
with the property\footnote{
Equivalent to: $\mathcal{P}$ is a refinement of the pullback of $\mathcal{Q}$ under $\phi$.
}: if points $x, x' \in X$ are in the same block of $\mathcal{P}$,
then $\phi(x), \phi(x') \in Y$ are  in the same block of $\mathcal{Q}$.

Carlsson and M{\'e}moli's program inspired us to ask whether 
there exists some notion of morphism between a pair of combinatorial data fusion problems,
such that data fusion algorithms are functors into the category of partitioned sets.

\subsection{Morphisms of graphs and of independence systems}
Suppose $(V, E)$ and $(V', E')$ are graphs, while $(V, \mathcal{S})$ and 
$(V', \mathcal{S}')$ are independence systems, also called abstract simplicial complexes.

\begin{definition}
An injective map $\phi:V \to V'$ is called a \textbf{graph morphism} if $\{\phi(u), \phi(v) \} \in E'$
whenever $\{u, v\} \in E$, and a \textbf{morphism of simplicial complexes} if $\phi(X) \in \mathcal{S}'$
whenever $X \in \mathcal{S}$. If map $\phi:V \to V'$ is injective and has both these properties,
call it a \textbf{weak morphism} of combinatorial data fusion problems:
\[
\phi: (V, E, \mathcal{S}) \to (V', E', \mathcal{S}').
\]
\end{definition}
Observe that the weight functions on $E$ and $E'$ play no part in this definition.

\subsection{Failure of functoriality: counterexamples}
Let $(V, E, w, \mathcal{F})$ and $(V', E', w', \mathcal{F}')$
be a pair of combinatorial data fusion problems, 
where $\mathcal{F}$ are the circuits of $\mathcal{S}$, and
$\mathcal{F}'$ are the circuits of $\mathcal{S}'$.
Let $\phi:V \to V'$ be a weak morphism of these combinatorial data fusion problems.

Suppose a partition $V_0, V_1, \ldots, V_{d-1}$ of $V$ forms a maximum weight matching
of $(V, \mathcal{S})$ under $w$, as in Section \ref{s:mwm}. Is it necessarily a refinement of
$\phi^{-1}(V_0'), \phi^{-1}(V_1'), \ldots, \phi^{-1}(V_{f-1}')$,
for some maximum weight matching
$V_0', V_1', \ldots, V_{f-1}'$ of $(V', \mathcal{S}')$ under $w'$,
by analogy with Carlsson and M{\'e}moli's paradigm?

The counterexamples of Figures \ref{f:inducedcounterex} and \ref{f:extraconstraintcounterex}
show two important kinds of mis-match between two such combinatorial data fusion problems,
linked by a weak morphism.

\begin{enumerate}

\item
Suppose $(V, E)$ is a \textbf{vertex-induced subgraph} of $(V', E')$, which means that $E$ is precisely
the edges in $E'$ with both endpoints in $V$. Suppose $\mathcal{F}' = \mathcal{F}$, 
so the only forbidden sets in $V'$ are the ones in $V$. Certainly the inclusion map
$\iota: V \to V'$ is a weak morphism.
Figure \ref{f:inducedcounterex} shows how, even with a single forbidden set $\{v_1, v_2, v_3\}$,
the unique maximum weight matching of $(V, \mathcal{S})$ is not a refinement of the 
unique maximum weight matching of $(V', \mathcal{S}')$.

\begin{figure}
\caption{ \textbf{INDUCED SUBGRAPH COUNTEREXAMPLE: }
\textit{An edge-weighted graph with five edges on $\{v_1, \ldots, v_5\}$ is shown
on the left. The only forbidden set is $\{v_1, v_2, v_3\}$.
In the optimum solution, edge $\{v_2, v_3\}$ is discarded, and the 
maximum weight matching of the allowable hypergraph is
$\mathcal{C}':= \{\{v_3, v_4\}, \{v_1, v_2, v_5\} \}$. 
For the induced subgraph on vertices $\{v_1, v_2, v_3\}$, shown on the right,
the maximum weight matching partitions the vertices into
$\mathcal{C}:= \{ \{v_1\}, \{v_2, v_3\} \}$.
Notice that $v_2$ and $v_3$ are in the same block of $\mathcal{C}$,
but in different blocks of $\mathcal{C}'$. So $\mathcal{C}$ is not a refinement of $\mathcal{C}'$.}
} \label{f:inducedcounterex}

\begin{center}
\scalebox{0.4}{\includegraphics{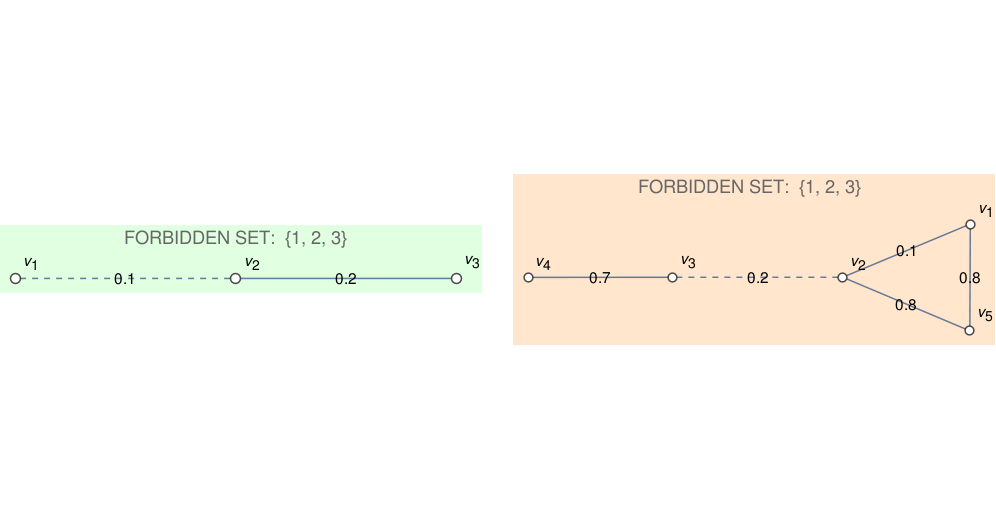}} 
\end{center}

\end{figure}

\item
Take the graphs $(V, E, w)$ and $(V', E', w')$ as identical, but $\mathcal{S} \subsetneq \mathcal{S}'$.
From the forbidden sets perspective, $\mathcal{F} \supsetneq \mathcal{F}'$
In Figure \ref{f:extraconstraintcounterex}, $|\mathcal{F}|=2$ (right side), while
$|\mathcal{F}'|=1$ (left side). The unique maximum weight matching of $(V, \mathcal{S})$ 
is not a refinement of the 
unique maximum weight matching of $(V', \mathcal{S}')$.

\begin{figure}
\caption{ \textbf{EXTRA CONSTRAINT COUNTEREXAMPLE: }
\textit{An edge-weighted graph with six edges on $\{v_1, \ldots, v_6\}$ is shown.
When $\{v_1, v_2, v_5\}$ is the only forbidden set, 
one edge is discarded (shown dashed), and the maximum weight matching of $(V, \mathcal{S}')$ is
$\mathcal{C}_1:= \{\{v_1, v_2\}, \{v_3, v_4, v_5, v_6\} \}$. 
If a second forbidden set $\{v_4, v_5\}$ is added,
then two edges are discarded, and the maximum weight matching of $(V, \mathcal{S})$ is
$\mathcal{C}_2:= \{ \{v_5, v_6\}, \{v_1, v_2, v_3, v_4\} \}$.
Notice that $\mathcal{C}_2$ is not a refinement of $\mathcal{C}_1$.} 
\textbf{MATROID COUNTEREXAMPLE: } 
\textit{The five solid edges on the left, and the four solid edges
on the right, are maximal elements of the independence system $\mathcal{E}$ defined
in Section \ref{s:msg}, when $\{v_1, v_2, v_5\}$ is the only forbidden set; since
$\mathcal{E}$ contains maximal elements of different sizes, it is not a matroid.
}
} \label{f:extraconstraintcounterex}

\begin{center}
\scalebox{0.4}{\includegraphics{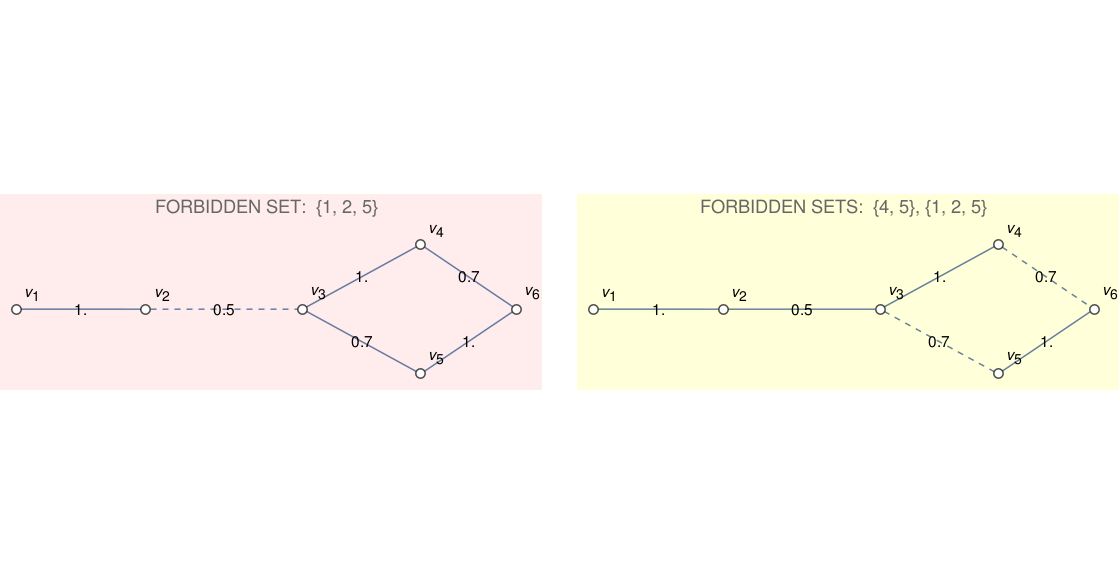}} 
\end{center}

\end{figure}

\end{enumerate}

%%%%%%%%%%%%%%%%%%%%%%%%%%%%%%%%%%%%%%%%%%%%%%%%%%%%%%%%%%%%%%%%%%%%%%%%%%%%%%%%%%%%%%%%%%%%
\section{Number of colors in an optimum coloring}

\subsection{Goal}
Before studying algorithms, let us take a look at the space of optimum solutions.
In this section we shall study the number of colors that are required in an optimal coloring. 

The number of colors (or equivalently the number of connected components in a subgraph, 
or hyperedges in an optimal perfect matching) does not appear in the combinatorial data fusion objective function.
However, monochromatic $G$ edges do appear in the objective function and the more colors used in a coloring of $H$
 the harder it is to maintain monochromatic edges in $G$. Additionally, some algorithmic techniques put a limit on the 
 number of allowable colors and/or have the number of colors appear in their run time. 
 
\subsection{Lack of transversality}
It is natural to wonder about the interaction between the forbidden sets of a combinatorial data fusion problem,
and the color classes of an optimal solution. For example, 
when every forbidden set has size at least $k$, and the number of colors in
an optimum hypergraph coloring does not exceed $k$, must
each color necessarily have a representative vertex in each forbidden set?
This is tautologically true if $k = 2$, but false in general for $k \geq 3$.

Figure \ref{f:nosdr} offers a counterexample.
Even when the forbidden sets are pairwise disjoint, it is possible that 
a maximum weight matching of the allowable hypergraph partitions
 the vertices into blocks, none of which contains a set of
distinct representatives\footnote{Such a set is called a transversal.} of all the forbidden sets.

\begin{figure}
\caption{ \textbf{TRANSVERSALITY AND 2-COLORING COUNTEREXAMPLE: }
\textit{An edge-weighted graph with 18 edges on $\{v_0, \ldots, v_9\}$ is shown.
The forbidden sets are $\{v_1, v_4, v_7\}$, $\{v_2, v_5, v_8\}$, and $\{v_3, v_6, v_9\}$,
shown in different colors.
The maximum weight matching of the allowable hypergraph partitions the vertices into blocks
$\{v_0, v_2, v_6, v_9\}$,  $\{v_3, v_4, v_7\}$, $\{v_1, v_5, v_8\}$, none of which
 intersects all three forbidden sets. The forbidden hypergraph is 2-colorable,
for example by mapping $v_1, v_2, v_3$ to color 0, and other vertices to color 1,
 but the optimum coloring uses 3 colors.}
} \label{f:nosdr}

\begin{center}
\scalebox{0.45}{\includegraphics{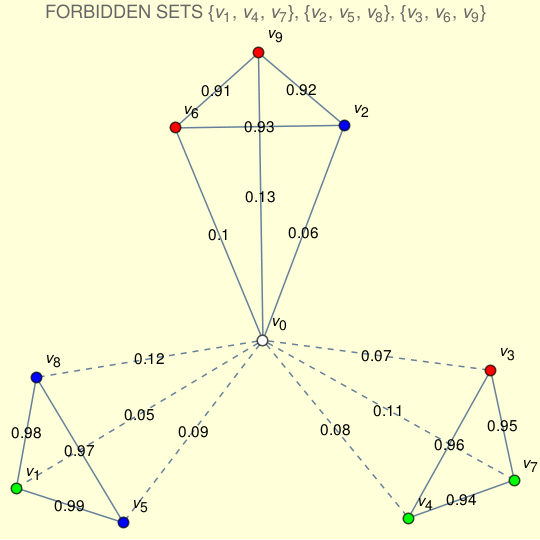}} 
\end{center}

\end{figure}

\subsection{Solution in the case of a single forbidden set}
Suppose there is a single forbidden set $F \subset V$ for the weighted graph $(V,E_0,w)$ 
and that $|F| = k \geq 2$. 
For each of the distinct non-empty subsets $K \subset F$ of size not exceeding $k/2$, 
we specialize the combinatorial data fusion problem for $(V, E_0, w, \{F\})$
to the case where the solution has two graph components, one of which contains $K$,
while the other contains $F \setminus K$. In the latter case, solve the
Minimum $\{s,t\}$ Cut problem on an augmented graph with vertex set $V \cup \{s, t\}$ where a source
node $s$ is joined to each $v \in K$ by an infinite weight edge, and a sink node $t$
is joined to each $v' \in F \setminus K$ by an infinite weight edge. 

Korte \& Vygen \cite[Ch. 8]{kor} describe the classical algorithms of Ford \& Fulkerson, and others,
for the equivalent Maximum Flow problem on the augmented graph. Having solved the problem to find a 
Minimum $\{s,t\}$ Cut, we associate with $K$ the total weight of the edges in the Minimum Cut
between $s$ and $t$. Finally choose the $K$ for which this weight is least, and declare
the edges outside this Minimum Cut to be an optimum solution.
We claim that:

\begin{proposition} \label{p:onef}
When there is only a single forbidden set $F$, with cardinality $|F| = k$,
the combinatorial data fusion problem $(V, E_0, w, \{F\})$ can be solved
by picking the least costly of $2^{k-1} - 1$
Minimum $\{s,t\}$ Cut problems on $V \cup \{s, t\}$.
\end{proposition}
Proposition \ref{p:onef} is established by the argument above, together with the fact that
 an optimum solution for $(V, E_0, w, \{F\})$ has exactly two components, which we now prove.

\begin{lemma} \label{l:oneforb}
When there is only one forbidden set, a maximum subgraph (Section \ref{s:msg}) for a
combinatorial data fusion problem $(V, E_0, w, \{F\})$ has exactly two components.
\end{lemma}

\textbf{Remark: } See Theorem \ref{t:numcolorsmax} for a stronger result.

\begin{proof}
Since $F \subset V$ is forbidden, a maximum subgraph cannot have one component.
We prove that a subgraph with more than two components cannot be optimal. 

Suppose that $E_1 \subset E_0$ and that $F$ is not connected in $G_1 = (V,E_1)$. If $G_1$ has more that $2$ connected components then let $C_0, C_1, \dots, C_k$ be the connected components of $G_1$ ordered such that $C_{1}$ and $C_2$ both intersect $F$. As $G=(V,E_0)$ is connected there must exist an edge $e \in E_0 \setminus E_1$ with exactly one endpoint in $C_0$. Let $E'_1 = E_1 \cup \{e\}$ and $G'_1 = (V,E'_1)$.  The graph $G'_1$ can be formed from $G_1$ be adding the edge $e$. The effect of adding $e$ to $G_1$  will be to merge $C_0$ with one of $C_1$, \dots,  $C_k$. Thus the vertices of $C_1$ and $C_2$ remain in different components of $G'_1$ and we have that $F$ is not connected in $G'_1$. Therefore $E'_1$ is a feasible solution with a larger objective function value than $E_1$. It follows that every optimal solution has exactly two components.
\end{proof}

\subsection{Upper bound on optimum number of colors}
Here is a stronger form of Lemma \ref{l:oneforb}. Pike et. al. \cite{pik} proved

\begin{theorem} \label{t:numcolorsmax}
When a combinatorial data fusion problem has $|\mathcal{F}|=b$
forbidden sets, then there exists an optimum coloring of 
the forbidden hypergraph $H=(V,\mathcal{F})$ which uses at most $t$ colors, where $2b \geq t^2 - t$.
\end{theorem}

\textbf{Examples: } Two colors suffice for $|\mathcal{F}|=2$; three colors suffice for $|\mathcal{F}| \in \{3, 4, 5\}$. 

We shall elaborate on the proof from \cite{pik}.

\begin{proof}
Let $t$ be the minimum number of colors needed for an optimal solution and let $\chi$ be an optimal coloring on the
$t$ colors $\{1,2,\dots, t\}$. Let $\chi_i$ denote the set of vertices which receive color $i$.

Suppose that there exists a pair of distinct colors $i$ and $j$ such that the set $\chi_i \cup \chi_j$ does not contain 
a forbidden set. Then the coloring $\chi'$ equal to $\chi$ with colors $i$ and $j$ merged into a single color would be a 
valid coloring of $H$. By construction the objective function value of $\chi'$ must be at least that of $\chi$.
Thus $\chi'$ would be an optimal coloring of $H$ on fewer colors than $\chi$, a contradiction. 

Therefore for every pair of distinct colors $\{i,j\}$ there exists a forbidden set $f_{i,j}$ contained in $\chi_i \cup \chi_j$. As $\chi$ is a valid
coloring of $H$, no $\chi_i$ contains a forbidden set. This implies that the $\binom{t}{2}$ forbidden sets $f_{i,j}$ are distinct 
and thus $b \geq \binom{t}{2}$ as required.
\end{proof}

\begin{figure}
\caption{ \textbf{TWO FORBIDDEN SETS: }
\textit{This illustrates the $b=2$ version of Theorem \ref{t:numcolorsmax}.
Cut 1 disconnects forbidden set $F$, while cut 2 disconnects forbidden set $F'$. This gives a 3-coloring.
However we may merge the color class above cut 2 with the color class to the left of cut 1, giving cut 0, shown on the right,
which is equivalent to a 2-coloring. Cut 0 maintains feasibility, since $F$ and $F'$ both remain disconnected,
while not increasing the weight of cut edges in the graph.
}
} \label{f:2col42fs}
\begin{center}
\scalebox{0.5}{\includegraphics{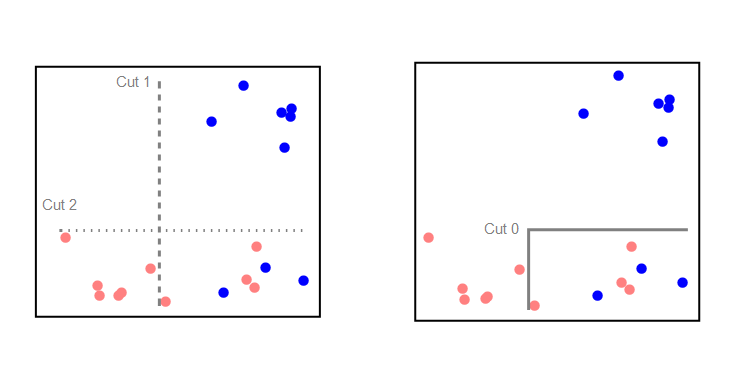}} 
\end{center}
\end{figure}

The following Corollary, which uses the technique of Theorem \ref{p:onef},
 extends the well known result that optimum multicut with two
source-sink pairs can be computed in polynomial time. We shall return to the case of 
two forbidden sets in Section \ref{s:2forbsetsgreedy}.

\begin{corollary}\label{c:2forbset}
When a combinatorial data fusion problem has exactly two forbidden sets, 
say $\mathcal{F}:=\{F,F'\}$,
an optimum coloring of $(V,\mathcal{F})$, with two colors, can be found using 
$p$ Minimum $\{s,t\}$ Cut computations on graphs with vertex set $V \cup \{s, t\}$
and $|F| + |F'|$ additional edges,
where $p$ is the number of partitions $\mathcal{P}_2$ of $F \cup F'$ into two subsets,
 both of which non-trivially intersect $F$ and $F'$.
\end{corollary}
\begin{proof}
Theorem \ref{t:numcolorsmax} shows that there is an optimum which is a 2-coloring.
Such an optimum partitions $V$ into $V_0 \cup V_1$, where $V_i \cap F \neq \emptyset$
and $V_i \cap F' \neq \emptyset$, for $i = 0, 1$.
To find such a partition, iterate through all partitions $\{A, B\} \in 
\mathcal{P}_2$; for each we connect a source $s$ to all vertices in $A$, and a sink $t$
to all vertices in $B$, where the new edges have infinite weight,
and find a Minimum $\{s,t\}$ Cut in the augmented graph on vertex set $V \cup \{s, t\}$. 
Having found a pair 
$\{A, B\} \in \mathcal{P}_2$ whose cut has minimum weight, set $V_0\supset A$ to 
be all vertices of $V$ in the same component as $s$, and set $V_1 \supset B$ to 
be all vertices of $V$ in the same component as $t$, respectively.
\end{proof}

\subsection{Minimum number of colors need not give optimum coloring} \label{s:mincolnotopt}

The number of colors used in an optimum hypergraph coloring
of a combinatorial data fusion problem may exceed the chromatic number
of the forbidden hypergraph. Figure \ref{f:nosdr} shows one example where 
the forbidden hypergraph is 2-colorable, but the optimum hypergraph coloring
uses three colors. Here is a simpler example of the same phenomenon.
Consider the path of length three shown in Figure \ref{f:path}, with forbidden sets 
$\{v_1, v_2\}, \{v_1,v_3,v_4\}$, and $\{v_2,v_3,v_4\}$. Two of the edges have weight 1, while
the bottom edge has weight $b > 2$. 
Different colors must assigned to $v_1$ and $v_2$, since $\{v_1, v_2\}$ is a forbidden set; without
loss of generality take $\chi(v_1)=0$ and $\chi(v_2)=1$. The optimum 2-coloring has
$\chi(v_4)=0$ and $\chi(v_3)=1$, shown on the lower left, with cost $b$.
The optimum 3-coloring has $\chi(v_4)=2 = \chi(v_3)$ shown on the lower right, with cost $2$.
Hence the best 2-coloring is $b/2$ times more expensive than the best 3-coloring, 
and this factor can be made arbitrarily large.

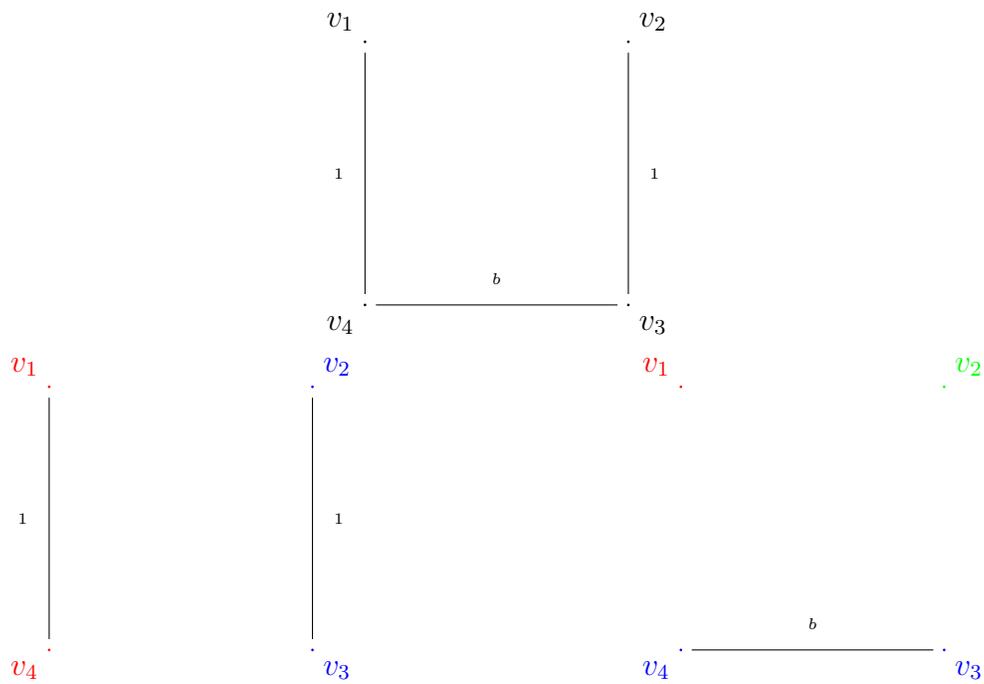
\begin{figure}
\caption{\textit{Example where the best 2-coloring is arbitrarily more expensive than the best 3-coloring}}\label{f:path}
\begin{center}
\begin{tikzpicture}[scale=1.75]
\node (v1) at (-1,1){};
\node (v2) at (1,1){};
\node (v3) at (1,-1){};
\node (v4) at (-1,-1){};

\fill (v1) circle (0.01) node [above left] {$v_1$};
\fill (v2) circle (0.01) node [above right] {$v_2$};
\fill (v3) circle (0.01) node [below right] {$v_3$};
\fill (v4) circle (0.01) node [below left] {$v_4$};

\draw (v2) -- (v3)--(v4) -- (v1);

%\node at (0,1.2) {\tiny{$1$}};
\node at (1.2,0) {\tiny{$1$}};
\node at (0,-0.8) {\tiny{$b$}};
\node at (-1.2,0) {\tiny{$1$}};
\end{tikzpicture}

\begin{tabular}{l r}
\begin{tikzpicture}[scale=1.75]
\node (v1) at (-1,1){};
\node (v2) at (1,1){};
\node (v3) at (1,-1){};
\node (v4) at (-1,-1){};

\fill [color = red] (v1) circle (0.01) node [above left] {$v_1$};
\fill [color = blue] (v2) circle (0.01) node [above right] {$v_2$};
\fill [color = blue] (v3) circle (0.01) node [below right] {$v_3$};
\fill [color = red](v4) circle (0.01) node [below left] {$v_4$};

\draw (v1) -- (v4);
\draw (v2) -- (v3);

%\node at (0,1.2) {\tiny{$1$}};
\node at (1.2,0) {\tiny{$1$}};
%\node at (0,-0.8) {\tiny{$b$}};
\node at (-1.2,0) {\tiny{$1$}};

\end{tikzpicture} &

\begin{tikzpicture}[scale=1.75]
\node (v1) at (-1,1){};
\node (v2) at (1,1){};
\node (v3) at (1,-1){};
\node (v4) at (-1,-1){};

\fill [color = red] (v1) circle (0.01) node [above left] {$v_1$};
\fill [color = green] (v2) circle (0.01) node [above right] {$v_2$};
\fill [color = blue] (v3) circle (0.01) node [below right] {$v_3$};
\fill [color = blue](v4) circle (0.01) node [below left] {$v_4$};

\draw (v4) -- (v3);

%\node at (0,1.2) {\tiny{$1$}};
%\node at (1.2,0) {\tiny{$1$}};
\node at (0,-0.8) {\tiny{$b$}};
%\node at (-1.2,0) {\tiny{$1$}};

\end{tikzpicture}\\
1. \textit{Two color solution, cut cost $b>2$} & 2. \textit{Three color solution, cut cost 2}\\
\end{tabular}
\end{center}
\end{figure}

\textbf{Open Problem: } \textit{Under what conditions on a 
combinatorial data fusion problem $(V, E, w, \mathcal{F})$
does an optimum hypergraph coloring of $(V,\mathcal{F})$ have exactly two colors?}

\textbf{Remark: } 
The literature on the question of when a hypergraph is 2-colorable includes
Seymour \cite{sey} and McDiarmid \cite{mcd}. The latter proves, using the Lov\'{a}sz Local Lemma, 
that a hypergraph, where every hyperedge contains at least $k$ points and meets at most $d$ other hyperedges,
 is 2-colorable if $e(d+2) \leq 2^k$.

%%%%%%%%%%%%%%%%%%%%%%%%%%%%%%%%%%%%%%%%%%%%%%%%%%%%%%%%%%%%%%%%%%%%%%%%%%%%%%%%%%%%%%%%%%%%

\section{The combinatorial data fusion problem for trees}

Let us move on to the study of solution methods. Begin with the case where $G$ is a tree.

\subsection{Set cover formulation}\label{s:ilpf}

When $(V, E, w)$ is an edge-weighted tree with positive weights, 
the combinatorial data fusion problem $(V, E, w, \mathcal{F})$
may be formulated as a set cover problem, without
requiring the number $d$ of colors to be specified in advance.

Each forbidden set $F$ will induce a linear constraint as follows.
Let $E_F \subseteq E$ denote 
the union, over all $v, v' \in F$, of the edges on the unique path from $v$ to $v'$ in the tree $(V, E)$.
The subtree\footnote{
We will revisit this subtree in Definition \ref{d:pathdiscrim}.
} $E_F$ may be constructed in time which is linear in $|E|$:
\begin{enumerate}

\item
Create an auxiliary vertex $v_o$, and auxiliary edges $E_o:=\{\{v_o, u\}, u \in F\}$.

\item
Let $E_F'$ denote the set of edges in the 2-core\footnote{
The maximal subgraph in which all vertices have degree at least $2$.
} of the graph
\[
(V \cup \{v_o\}, E \cup E_o).
\]
Observe that, for $u, u' \in F$, every edge on a path $u$ to $u'$ is part of a cycle
passing through $v_o$, $u$, and $u'$, and hence is in $E_F'$. However an edge not on any
such path will lead to a leaf vertex in $(V \cup \{v_o\}, E \cup E_o)$, and will therefore
be missing from the 2-core.

\item
The desired edge set $E_F$ is given by:
\(
E_F = E_F' \setminus E_o.
\)

\end{enumerate}
\textbf{Constraint Matrix: }Suppose $\{e_1, e_2, \ldots, e_m\}$ are the edges of the tree, and 
$\{F_1, F_2, \ldots, F_b\}$ are the forbidden sets. 
Encode all $b$ constraints in a $b \times m$ matrix $A:=(a_{i,j})$,
where $a_{i,j} \in \{0, 1\}$, and $a_{i,j} = 1$ precisely when edge $e_j$ belongs to $E_{F_i}$,
in other words when removal of $e_j$ disconnects the vertices in $F_i$.

The combinatorial data fusion problem $(V, E, w, \mathcal{F})$ now reduces to a 
minimum weight \textbf{set cover} problem:
we seek a subset $K \subset E = \{e_1, e_2, \ldots, e_m\}$ of minimum weight
so that $K \cap F_i \neq \emptyset$ for every forbidden set $F_i$. 
Here is the set cover formulation.

\begin{proposition} \label{p:hitset}
When the graph $(V,E)$ is a tree, an optimum set of deleted edges in the
combinatorial data fusion problem $(V, E, w, \mathcal{F})$ 
is given by $\{e_j \in E: x_j = 1\}$ in the set cover problem: minimize
\begin{equation} \label{e:hitwt}
\sum_{j = 1}^m w(e_j) x_j
\end{equation}
for variables $x_j \in \{0, 1\}$, over the constraints
\begin{equation} \label{e:hitconstr}
\sum_{j = 1}^m a_{i,j} x_j \geq 1, \quad i = 1, 2, \ldots,b.
\end{equation}
\end{proposition}

%\textbf{Remark: } See Corollary \ref{c:muculinear} for an even better result for multiway cut on trees.

\begin{proof}
Take any forbidden set $F$. Under the constraints, 
there is some edge $e \in E_F$ for which $x_e = 1$, and that edge is on
the unique path between some pair of vertices $v, v' \in F$. Hence $v$ and $v'$ are not in the same
component of the forest $(V, E_1)$, where $E_1:=\{e \in E: x_e = 0\}$. Therefore $F$ cannot be contained
in a component of $(V, E_1)$. Thus the constraints of the integer linear program precisely
identify the subsets $E_1 \subset E$ such that no component of $(V, E_1)$ contains any forbidden set.
\end{proof}

\subsection{Practical implementation}

\subsubsection{Multicut problem on trees}
Williamson and Shmoys \cite[Exercise 7.2]{shm} gives a 2-approximation algorithm based on Primal Dual Schema
for the case where every forbidden set has cardinality 2, i.e the multicut problem on trees.

\subsubsection{Set cover approximations}\label{s:setcovappr}
The Set Cover problem (\ref{e:hitwt}), (\ref{e:hitconstr}) is treated at length in
Williamson \& Shmoys \cite[Ch. 1, 7]{shm} and in Vazirani \cite{vaz}.
Available algorithms include the following.

\begin{itemize}
\item
The \textbf{Primal-Dual} algorithm gives a $c$-factor approximation, where $c$ is the maximum of 
the column sums of the matrix $A$, or in other words the maximum number of forbidden sets which
can be disconnected by the removal of a single edge.
\item
In the \textbf{Greedy Set Cover} algorithm, $J_n \subset \{1,2, \ldots, m\}$ will denote the
indices of the edges deleted after $n$ steps, initialized at $J_0 = \emptyset$.
Stop when every forbidden set has been disconnected. After $n-1$ steps, the next iteration goes as follows:
\begin{enumerate}
\item
For each $j \notin J_{n-1}$, score edge $e_j$ by $w_j$ divided by the number of
connected forbidden sets which will be disconnected if $e_j$ is removed.

\item
Select $e_j$ with minimum score, using random tie-breaking, and set $J_n:=J_{n-1} \cup \{j\}$.
Reduce the collection of connected forbidden sets accordingly.

\end{enumerate}

Taking $c$ as the maximum of the column sums of the matrix $A$ as above,
\cite[Section 1.6]{shm} proves that $\max J_n$ gives an $H_c$-approximation, where
\begin{equation} \label{e:harmonic}
H_c:= \sum_{s=1}^c \frac{1}{s} \approx \log{c}.
\end{equation}
This is a better bound than that of the Primal-Dual algorithm.

\end{itemize}

\subsubsection{Exact solution through linear programming relaxation}
Call a matrix A \textbf{totally unimodular} if each subdeterminant of A is
0, 1, or -1. According to a theorem of Ghouila and Houri proved in \cite[Ch. 5]{kor},
a $b \times m$ integer matrix $A:=(a_{i,j})$ is totally unimodular if and only if, for every
$R \subset \{1, 2, ,\ldots, b\}$, there is a partition $R = R_1 \cup R_2$ such  that
\begin{equation} \label{e:tunim}
\sum_{i \in R_1} a_{i,j} - \sum_{i \in R_2} a_{i,j} \in \{-1, 0, 1\}, \quad j = 1, 2, \ldots, m.
\end{equation}

Consider the linear programming relaxation of the integer linear program of Proposition \ref{p:hitset}.
Instead of $x_j \in \{0, 1\}$, take $x_j \geq 0$ and as before
\[
A \mathbf{x} \geq \mathbf{1}
\]
where $\mathbf{1}$ is the vector whose $b$ components are all 1.

\begin{corollary}
Suppose the constraint matrix $A$ satisfies (\ref{e:tunim}), and is therefore totally unimodular.
Then the linear programming relaxation of (\ref{e:hitwt}), (\ref{e:hitconstr}) has an optimum integer solution,
which is optimum for the combinatorial data fusion problem on the tree.
\end{corollary}

\begin{proof}
See \cite[Ch. 5]{kor} for material on total dual integrality when $A$ is totally unimodular.
\end{proof}

%Further material on linear programming methods for combinatorial data fusion
%will appear in \cite{cdf2}.

%%%%%%%%%%%%%%%%%%%%%%%%%%%%%%%%%%%%%%%%%%%%%%%%%%%%%%%%%%%%%%%%%%%%%%%%%%%%%%%%%%%%%%%%%%%%
\section{Data fusion based on Gomory-Hu trees}

Let us consider solution methods when $G$ is not a tree. 

\subsection{Definition of Gomory-Hu tree } \label{s:gomhu}

Recall the definition (Schrijver \cite[Ch. 15]{sch})
of a Gomory-Hu tree $(V, E_T, \omega_T)$ for the weighted graph
$G:=(V, E_0, w)$. This is a tree on the same vertex set as $G$, 
with edge weight function $\omega_T$, such that for any
two vertices $v, v' \in V$, the cost $c_{v, v'}$ of a minimum cut
 (i.e. the total weight of deleted edges)
between any distinct $v$ and $v'$ in $G$ equals the minimum of $\omega_T(e)$ over edges $e$
along the unique path from $v$ to $v'$ in $(V, E_T)$. Note that
\begin{enumerate}
\item
The tree edges $E_T$ need not be a subset of graph edges $E_0$.
\item
The construction is not canonical. For example, given four vertices $v_1, v_2, v_3, v_4$ in $V$,
there may be one Gomory-Hu tree in which removing the lowest weight edge on the path between $v_1$
and $v_3$ also disconnects $v_2$ from $v_4$, and another Gomory-Hu Tree for the same $G$ in which
it does not.

\end{enumerate}

See Vazirani \cite[Ch. 4.2]{vaz} for a description of how to construct the Gomory-Hu Tree,
and an explanation of its uses in multicut problems.
We tested the \texttt{JGraphT} implementation of \texttt{GusfieldGomoryHuCutTree} on random graphs with
comparable vertex degree distributions, from 90 to 98000 edges. The run time grew at a sub-quadratic rate
in the number of edges.

\subsection{Gomory-Hu tree for example of Figure \ref{f:greedynonopt}}
When a single edge is removed from a Gomory-Hu tree,
the weight of that edge corresponds to the cost (\ref{e:hgcof}) of the 2-coloring 
which assigns color 0 to one tree component, and color 1 to the other tree component.

However when two edges $e, e'$ are removed from a Gomory-Hu tree, breaking it into three components,
the cost (\ref{e:hgcof}) of the corresponding 3-coloring could be less than $\omega_T(e) + \omega_T(e')$.
This is because some edges
of $G$ could be counted both towards $\omega_T(e)$ and towards $\omega_T(e')$.
This is illustrated in Figure \ref{f:gomoryhu}, where we remove two edges from a
Gomory-Hu Tree for the graph shown in  Figure \ref{f:greedynonopt}.
Here the optimum coloring of the forbidden hypergraph is a
2-coloring (Figure \ref{f:greedynonopt}), whereas a 3-coloring of the
Gomory-Hu tree was needed to disconnect both forbidden sets. The caption explains how two of these colors
can be merged, giving the optimum solution.

\begin{figure}
\begin{center}
\begin{tikzpicture}[scale=1.75]
\node (v1) at (1,0){};
\node (v2) at (-1,0){};
\node (v3) at (0,-1){};
\node (v4) at (0,0){};
\node (v5) at (-1,-1){};
\node (v6) at (-2,0){};
\node (v7) at (0,2){};
\node (v8) at (0,1){};
\node (v9) at (0,-2){};
\node (v10) at (1,-1){};
\node (v11) at (2,0){};

\fill (v1) circle (0.01) node [above] {$v_1$};
\fill (v2) circle (0.01) node [below] {$v_2$};
\fill (v3) circle (0.01) node [right] {$v_3$};
\fill (v4) circle (0.01) node [below right] {$v_4$};
\fill (v5) circle (0.01) node [below left] {$v_5$};
\fill (v6) circle (0.01) node [left] {$v_6$};
\fill (v7) circle (0.01) node [above] {$v_7$};
\fill (v8) circle (0.01) node [right] {$v_8$};
\fill (v9) circle (0.01) node [below] {$v_9$};
\fill (v10) circle (0.01) node [below] {$v_{10}$};
\fill (v11) circle (0.01) node [right] {$v_{11}$};

\draw [thick] (v6) -- (v2) -- (v4)--(v1) -- (v11);
\draw [thick] (v7) -- (v8) -- (v4) -- (v3) -- (v9);
\draw [thick] (v1) -- (v10);
\draw [thick] (v4) --(v5);

\node at (-1.5,-0.1) {\tiny{$1.56$}};
\node at (-0.5,-0.1) {\tiny{$2.63$}};
\node at (0.5,-0.1) {\tiny{$2.47$}};
\node at (1.5,-0.1) {\tiny{$0.37$}};
\node at (0.2,-0.5) {\tiny{$3.06$}};
\node at (0.2,-1.5) {\tiny{$1.33$}};
\node at (0.2,0.5) {\tiny{$1.29$}};
\node at (0.2,1.5) {\tiny{$1.31$}};
\node at (-0.5,-0.7) {\tiny{$2.05$}};
\node at (1.2,-0.5) {\tiny{$0.73$}};
\end{tikzpicture}
\end{center}
\caption{
\textbf{Gomory-Hu tree} \textit{computed from the graph of Figure \ref{f:greedynonopt}, whose
 forbidden sets are $\{v_1, v_2, v_3\}$ and $\{v_5,  v_8\}$. 
By inspection, the least costly way to disconnect both forbidden sets in this Gomory-Hu tree 
is to remove two tree edges, $\{v_4,v_8\}$ and $\{v_1,v_4\}$. 
The total weight of graph edges to be removed seems to be $1.29+2.47= 3.76$, but
 the graph edge $\{v_1, v_7\}$ is removed twice, so the true cost is $3.76 - 0.56 = 3.2$.
This gives a 3-coloring of the forbidden hypergraph, of which $\{v_1, v_{10}, v_{11}\}$ is
one color class, and $\{v_7, v_8\}$ is another. These two color classes can be merged, because their
union contains no forbidden set. Thus the graph edge $\{v_1, v_7\}$ is restored, and
the total weight of removed edges drops to $3.20 - 0.56 = 2.64$, which coincides with
the optimal solution for the original graph of Figure \ref{f:greedynonopt}.
So the Gomory-Hu tree algorithm plus a merger yields the optimum solution in this case.
}
}\label{f:gomoryhu}
\end{figure}
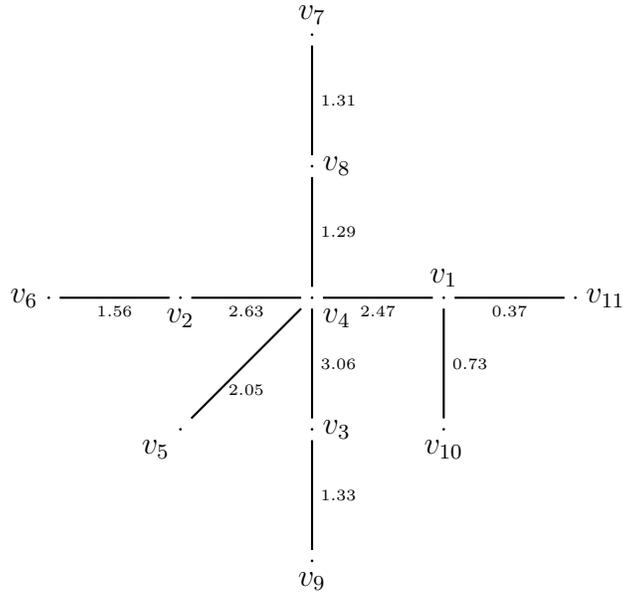

\subsection{Single forbidden set: why Gomory-Hu tree succeeds}

The question arises as to whether a Gomory-Hu tree $(V, E_T, \omega_T)$ may offer a better way to
solve the combinatorial data fusion problem for
the weighted graph $(V, E_0, w)$, than Proposition \ref{p:onef},
when there is just one forbidden set $F \subset V$.

When $|F|= 2$, it provides an optimum solution,
by definition. 
It would appear that when $|F| \geq 3$, the structure of the Gomory-Hu Tree 
disallows some partitions $K, F \setminus K$ from being considered. Suppose $F:=\{v_0, v_1, v_2\}$,
and the optimal partition of $V$ places $v_1$ in one component and
$\{v_0, v_2\}$ in the other. In the specific Gomory-Hu tree which is selected,
the unique path from $v_0$ to $v_2$ may pass through $v_1$. Hence the only partitions
whose cost is visible in this Gomory-Hu tree are (a) $v_0$ in one component,
$\{v_1, v_2\}$ in the other, and (b) $\{v_0, v_1\}$ in one component,
$v_2$ in the other. Despite this objection, Proposition \ref{p:gomhuone}
shows that optimum combinatorial data fusion can be derived from any Gomory-Hu tree.

\begin{definition}\label{d:pathdiscrim}
Given a tree $T:=(V, E)$ and a subset $F \subset V$ with $|F| \geq 3$,
apply the following iterative algorithm: 
whenever a leaf vertex is not an element of $F$, delete it and the incident edge.
The terminal state is a tree $T':=(V', E')$, which we call the \textbf{$F$-bounded subtree},
all of whose leaf vertices are elements of $F$. If all elements of $F$ are leaves
of the $F$-bounded subtree, we say that the tree $T$ \textbf{path-discriminates} the 
vertex set $F$.
\end{definition}

The important feature of a tree $T$ which path-discriminates $F$ is that
there do not exist three vertices $v, v', v'' \in F$ such that the unique path in $T$ from
$v$ to $v'$ passes through $v''$.

\begin{proposition} \label{p:gomhuone}
Consider a combinatorial data fusion problem on $G:=(V, E_0, w)$ in which there is a  
single forbidden set $F$ with $|F| \geq 3$.
Suppose $T:=(V, E_T, \omega_T)$ is any Gomory-Hu tree for $G$, which may or may not
 path-discriminate the vertex set $F$, in the sense of Definition \ref{d:pathdiscrim}.
A 2-coloring $\chi \to \{0, 1\}$ obtained as follows is an optimum solution:
\begin{enumerate}
\item
Construct the $F$-bounded subtree $T':=(V', E_T')$ of $T$, whose leaf vertices
are a subset of $F$ by Definition \ref{d:pathdiscrim}.
\item
Identify an edge $e_o \in E_T' \subset E_T$ of minimum weight. By construction, removal of this edge
from $E_T$ breaks the Gomory-Hu tree into two vertex components 
$V_0$ and $V_1$, neither of which contains $F$.
\item
Return the 2-coloring $\chi$ with value 0 on $V_0$, and value 1 on $V_1$. 
\end{enumerate}
Then $\omega_T(e_o)$ is the total weight of edges of $G$ with differently colored
end points under $\chi$. 

\end{proposition}

\begin{proof}
Removal of any $e_o \in E_T' \subset E_T$ partitions $V'$ into two components, both
of which contain at least one leaf vertex of $T'$, and hence at least one element of $F$. 
Neither component contains $F$, so the proposed 2-coloring $\chi$ is not constant on $F$,
and is indeed a proper coloring. 

By Lemma \ref{l:oneforb}, an optimum solution
of this combinatorial data fusion problem is a 2-coloring of $V$. Hence it suffices to prove that
there is no proper 2-coloring $\chi'$ of $(V, \{F\})$ for which the total weight of edges of $G$
 with differently colored end points is less than $\omega_T(e_o)$.

Suppose $\chi'$ is a proper coloring of $(V,\{F\})$ that partitions $V$ into $V_0'$ and $V_1'$.
 Let $h(u,v)$ denote the weight of a minimum cut between vertices $u, v \in V$ in the graph $G$,
  and consider the value:
\[
h_*: = \min \{h(u, v), \quad u \in F \cap V_0', \quad v \in F \cap V_1'\}.
\]
There are distinct vertices $u_0 \in F \cap V_0'$ and $u_1 \in F \cap V_1'$ for which
$h(u_0, u_1)$ attains this minimum value $h_*$. The vertices $u_0$ and $u_1$ are in the $F$-bounded
 subtree\footnote{
$u_0$ and $u_1$ do not need to be leaf vertices in $T'$; the path-discrimination
assumption is unnecessary.
} $T':=(V', E_T')$. Therefore, by the properties of a Gomery-Hu tree, $h(u_0,u_1)$ is the minimum weight of the 
edges on the $(u_0,u_1)$-path in $T'$, denote this edge of $T'$ by $e'$.  
It follows that
\[
h_* = h(u_0,u_1) = \omega_T(e') \geq \omega_T(e_o),
\]
where edge $e_o \in E_T' \subset E_T$ has minimum weight. However, $h_*$ is also a lower bound on 
the total weight of edges of $G$ with differently colored end points under the coloring $\chi'$.
Hence $\chi'$ gives a deleted edge weight at least as large as that of $\chi$, as claimed.
\end{proof}

\subsection{Limitations on power of Gomory-Hu tree in general}

Suppose $\chi:V \to \{1,2,\ldots, t\}$ is an optimum forbidden hypergraph coloring for a combinatorial
 data fusion problem $(V, E, w, \mathcal{F})$ where $|\mathcal{F}|=b \geq \binom{t}{2}$;
existence of an optimum $\chi$ with such a range follows from
Theorem \ref{t:numcolorsmax}. Take a Gomory-Hu tree 
$T:=(V, E_T, \omega_T)$ for $G:=(V, E, w)$. We would like to compare the cost of discarded edges
in $T$ with those discarded in $G$, as in (\ref{e:hgcof}), i.e. to compare:
\[
C_G(\chi):=\sum_{\substack{
 \{u, v\} \in E \\
 \chi(u) \neq \chi(v)}
} w(\{u, v\}); \quad
C_T(\chi):=\sum_{\substack{
 \{u, v\} \in E_T \\
 \chi(u) \neq \chi(v)}
} \omega_T(\{u, v\}).
\]
Proposition \ref{p:gomhuone} ensures that if $b = 1$ then for each Gomery-Hu tree $T$ there exists
 an optimal $2$ coloring such that $C_G(\chi) = C_T(\chi)$. However, as the next proposition 
 shows, the ratio between $C_G(\chi)$ and $C_T(\chi)$ is not bounded even for 
the multicut problem ($b \geq1$, but all forbidden sets of size $2$).

\begin{proposition} \label{p:ugcght}
Assuming the Unique Games Conjecture (see \cite[Ch. 13 ]{shm}), there is no
constant $\alpha \geq 1$ such that, for every optimum coloring $\chi$ for
a multicut problem on a graph $G$, there is a Gomory-Hu tree
$T$ for $G$ such that
\begin{equation} \label{e:treecompare}
C_T(\chi) \leq \alpha C_G(\chi)
\end{equation}
unless \textbf{P = NP}.
\end{proposition}
 
\begin{proof}
Restrict to the multicut problem, meaning that $|F|=2$ for all $F \in \mathcal{F}$.
Let $\chi$ be an optimal coloring of $(V,\mathcal{F})$ with respect to $G$. 
Select a Gomory-Hu tree $T:=(V, E_T, \omega_T)$ for $G$, and let
$\psi:V \to \{1,2,\ldots, b\}$ be an optimum coloring of $(V, \mathcal{F})$
with respect to $T$ (i.e. a solution to the combinatorial data fusion problem with $G$ replaced by $T$). 
Williamson and Shmoys \cite[Exercise 7.2]{shm} gives a 2-approximation algorithm 
for the multicut problem on trees. That means that we may construct
$\psi':V \to \{1,2,\ldots, b\}$ so that
\[
C_T(\psi') \leq 2 C_T(\psi) \leq 2 C_T(\chi) .
\]
The second inequality is valid because $\chi$ is also a coloring (possibly not an optimum)
with respect to $T$. If the inequality (\ref{e:treecompare}) were valid for some $\alpha \geq 1$,
then
\[
C_T(\psi') \leq 2 \alpha C_G(\chi),
\]
In other words, the coloring $\psi'$ gives a $2 \alpha$-approximation to the multicut problem.
According to \cite[Theorem 8.10]{shm}, assuming the Unique Games Conjecture, such an 
approximation cannot exist unless \textbf{P = NP}.
\end{proof}

\section{Greedy data fusion algorithm based on Gomory-Hu tree} 
Despite the impossibility of constant-factor approximation, demonstrated in Proposition \ref{p:ugcght}, 
we shall attempt to use the Gomory-Hu tree to approximate combinatorial data fusion.
Experiments reported in Section \ref{s:ghuxpts} show surprisingly good performance.

The following algorithm, which we have implemented in Java on top of the \texttt{JGraphT} library,
 combines the idea of Proposition \ref{p:gomhuone} with the Greedy
Set Cover algorithm of Section \ref{s:setcovappr}.

\subsection{Algorithm}\label{s:greedygomhu}
A combinatorial data fusion problem $(V, E, w, \mathcal{F})$ is given.
\begin{enumerate}

\item
Compute a Gomory-Hu tree $T:=(V, E_T, \omega_T)$ for
$G:=(V, E, w)$. 

\item
Compute, for each forbidden set $f \in \mathcal{F}$, the edge set $E_f \subset E_T$ of 
the $f$-bounded subtree of $T$ (see Definition \ref{d:pathdiscrim}), for example by the method described in Section \ref{s:ilpf}.
Initialize Boolean arrays $(y_f, f \in \mathcal{F})$ at 1, and $(x_e, e \in E_T)$ at zero.
Subsequently $y_f=1$ will mean that forbidden set $f$ is connected, and $x_e = 1$
will mean that edge $e$ has been deleted from $T$.

\item
\textbf{Iterative Splitting: }
Continue until $y_f = 0$ for all $f \in \mathcal{F}$, i.e. all forbidden sets have been disconnected.

\begin{enumerate}
\item 
For each undeleted tree edge, i.e. $e \in E_T$ with $x_e = 0$, compute
the number of forbidden sets which will be newly disconnected if $e$ is removed:
\begin{equation} \label{e:phi_e} 
\phi(e):=\sum_{f \in \mathcal{F}} y_f 1_{  \{e \in E_f  \} }.
\end{equation}
\item
Select (with random tie-breaking) an undeleted $e \in E_T$ for which
the ratio $\omega_T(e)/\phi(e)$ is minimum, as in Greedy Set Cover.

\item
Delete this $e$ by setting $x_e = 1$, and disconnect those $f$ for which $e \in E_f$
by setting $y_f = 0$.

\end{enumerate}

\item \textbf{Color Mergers: }
If $d$ is the number of steps needed to disconnect
every forbidden set, then $d$ edges of the Gomory-Hu tree
have been removed, leaving a forest with $d+1$ components.
This corresponds to a $(d+1)$-coloring $\chi: V \to \{0, 1, \ldots, d\}$ of $(V,\mathcal{F})$.

\begin{enumerate}
\item
Construct an edge-weighted graph $M=(V(M), E(M))$ whose nodes are the color classes:
\[
V(M):= \{C_0, C_1, \ldots, C_d\}; \quad C_j:=\chi^{-1}(j) \subset V.
\]
The weight on an edge $\{C_i, C_j\}$ is defined as the weight of graph edges in 
$G:=(V, E, w)$ which would be restored if these two color classes were merged,
and where merging would not monocolor a forbidden set:
\[
\theta(\{C_i, C_j\}):=\sum_{u \in C_i, v \in C_j} w(\{u,v\}).
\]
Take the edge set $E(M)$ to be those pairs of color classes with non-zero weight.

\item
We wish to merge color classes without their merger including a forbidden set.
This is a combinatorial data fusion sub-problem $(V(M), E(M), \theta, \mathcal{C})$
where a collection of colors forms a forbidden set in $\mathcal{C}$ if their merger
would cause some forbidden set in $\mathcal{F}$ to become monochromatic.

\item
For a small number of colors, this sub-problem can be solved
exhaustively, by considering all possible color mergers. If $E(M)$ is too large for exhaustion,
perform color mergers in a greedy  way, with those of highest $\theta$ weight first.

\item
If $m+1$ mergers are possible, the final number of tree components is $d-m$.
Apply a different color to each component, giving a $d-m$-coloring of the 
original forbidden hypergraph.
\end{enumerate}

\end{enumerate}
%Question: How much better does this algorithm become when we regenerate the GH-Tree after each deletion?

We are not yet able to present a satisfactory analysis of the approximation
properties of the algorithm of Section \ref{s:greedygomhu}. 
See Table \ref{comparetab} for experimental results.
Here is a weak bound;

\begin{proposition} \label{p:2approx}
In the case of a combinatorial data fusion problem with exactly $b$ forbidden sets, 
the greedy algorithm of Section \ref{s:greedygomhu}
provides a coloring whose cost is no more than $b$ times the cost of an optimum coloring.
\end{proposition}

\textbf{Remark: } The approximation for multicut given in \cite[Theorem 8.9]{shm} suggests
that there may exist a $c \log{(b+1)}$-approximation algorithm,
where $c$ is a constant depending on the maximum or minimum cardinality of any forbidden set.

\begin{proof}
Let $C$ denote the cost of an optimum coloring. Let $e_F \in E_T$ be a tree edge which is of minimum cost
among those which disconnect $F$, for each forbidden set $F$. Then
\[
\omega_T(e_F) \leq C, \quad \forall F \in \mathcal{F}
\]
because in particular the optimum coloring disconnects $F$. Suppose the iterative splitting
part of the algorithm deletes tree edges $e_1, e_2, \ldots, e_m$, where necessarily
\[
\sum_{j=1}^m \phi(e_j) = b,
\]
($\phi(e)$ defined as in (\ref{e:phi_e}))
which is the number of forbidden sets to be split. For any $j \in \{1,2,\dots,m\}$, $\omega_T(e_j)$ does not exceed the sum
of those values of $\omega_T(e_F)$ for which forbidden set $F$ is split at step $j$, which is bounded 
above by $C \phi(e_j)$.
The total weight of edges discarded by the iterative splitting of the greedy algorithm does not exceed
\[
\sum_{j=1}^m \omega_T(e_j) \leq C \sum_{j=1}^m \phi(e_j) = b C.
\]
The merge phase can only lessen the weight of discarded  edges.
Thus the desired factor $b$ approximation is established.
\end{proof}

\subsection{Greedy data fusion: two forbidden sets}\label{s:2forbsetsgreedy}
We already analysed the case of exactly two forbidden sets $\{F, F'\}$ in Corollary \ref{c:2forbset}.
We will use this case to shed light on the greedy algorithm.
Recall that $E_F$ refers to the minimal set of edges of the Gomory-Hu tree which connect the vertices of $F$.
Thus removal of any one edge in $E_F$ disconnects $F$ in the tree.
There are two possibilities.

\begin{figure}
\caption{\textbf{Counterexamples to the optimality of the greedy algorithm when $|\mathcal{F}|=2$: }
\textit{Two examples where the forbidden sets are $F:=\{s, s'\}$ and $F':=\{t, t'\}$.
\textbf{Left: }
Graph itself is a tree. $E_F$ and $E_{F'}$ have an edge in common.The optimum is to remove the edge of
weight 6, which disconnects both $F$ and $F'$. However the greedy algorithm
removes the edges of weight 2 and 5, with total cost 7, which is worse.
\textbf{Upper Right: }Graph on five vertices with six edges. 
\textbf{Lower Right: }Its Gomory-Hu tree, in which $E_F$ and $E_{F'}$ are disjoint. 
The greedy algorithm on the Gomory-Hu tree removes the tree edges with weights 9 and 10.
Back in the original graph, this leads to a 2-coloring where edges of weight $8 + 9 = 17$
are cut. The optimum is to cut graph edges of weight $7 + 7 = 14$.
}
} \label{f:greedyfailson2fs}

\begin{center}
\scalebox{0.25}{\includegraphics{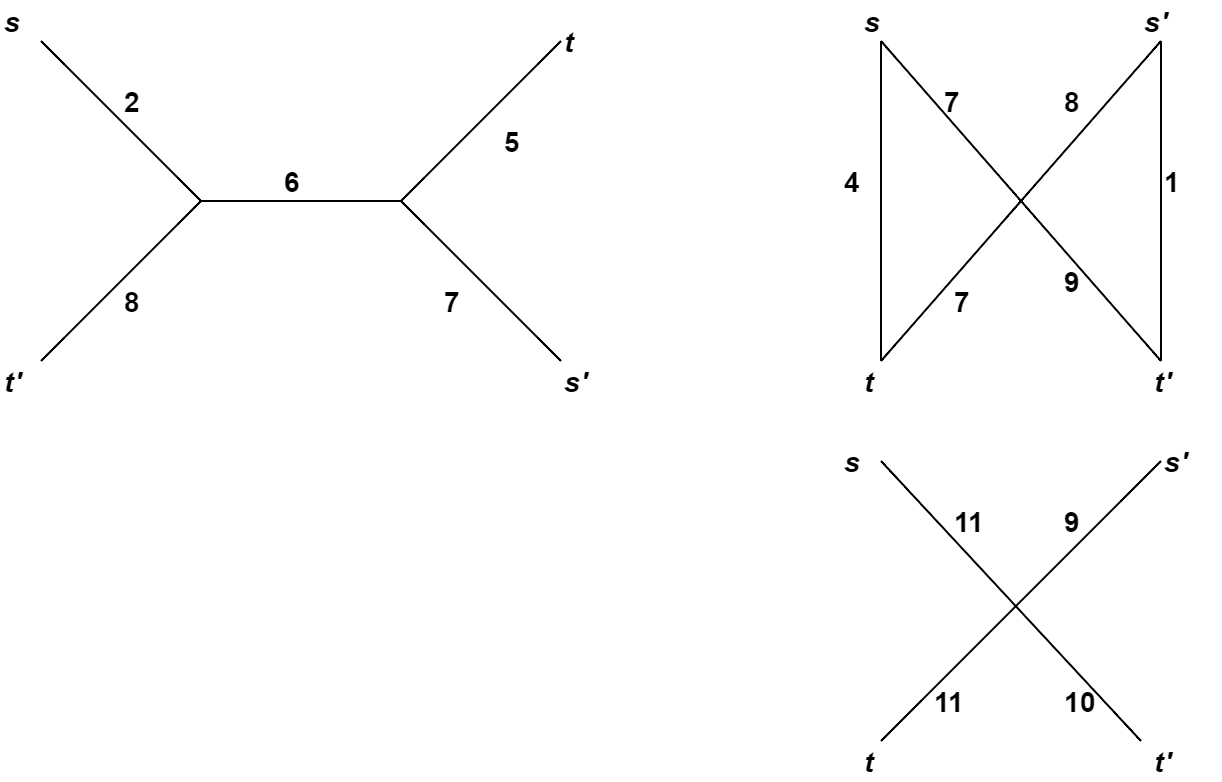}} 
\end{center}

\end{figure}

\begin{itemize}
\item
\textbf{$E_F$ and $E_{F'}$ are disjoint: } In this case there is no single edge of the 
Gomory-Hu tree whose removal disconnects both $F$ and $F'$.
\begin{enumerate}
\item
Figures \ref{f:greedynonopt} and \ref{f:gomoryhu} illustrate one possible outcome of the algorithm.
The two edges removed from the Gomory-Hu tree are $\{v_4,v_8\}$ and $\{v_1,v_4\}$,
each of which disconnects one forbidden set. At the end of the iteration, there are three
components, two of which can be merged, to give the optimum solution illustrated in
Figure \ref{f:greedynonopt}.

\item
The graph on the right in Figure \ref{f:greedyfailson2fs} shows how the greedy algorithm fails.
The greedy algorithm fails to take account of the extra weight of edges (4 or 1 in this case)
which are restored when two of the colors are merged.

\end{enumerate}

\item
\textbf{$E_F$ and $E_{F'}$ have an edge in common: } 
Suppose there is some edge of the 
Gomory-Hu tree whose removal disconnects both $F$ and $F'$.
Let $e_0 \in E_F$, $e_1 \in E_{F'}$, and $e_2 \in E_F \cap E_{F'}$
each have least tree weight among those in their respective edge sets.
Possibly they are not distinct. Evidently
\(
\max { \{\omega_T(e_0), \omega_T(e_1) \} } \leq \omega_T(e_2).
\)
If
\[
\omega_T(e_2) <  2 \min{ \{ \omega_T(e_0), \omega_T(e_1) \} },
\]
then the criterion of step 3(b) above chooses $e_2$ as the first edge to delete.
There are four cases to consider:
\begin{enumerate}
\item
If $2 \min{ \{\omega_T(e_0), \omega_T(e_1) \} }  >  
\omega_T(e_2) = \max { \{\omega_T(e_0), \omega_T(e_1) \} }$, 
then the deletion
of $e_2$ partitions the vertex set $V$ into two components, neither of which
contains $F$ or $F'$, and this is an optimum 2-coloring as provided by Corollary \ref{c:2forbset}.
\item
If  $2 \min{ \{\omega_T(e_0), \omega_T(e_1) \} }  >  \omega_T(e_2)  > \max { \{\omega_T(e_0), \omega_T(e_1) \} }$, then $e_2$ is the first edge
to be deleted by the greedy algorithm, but this need not be optimal (e.g. when $\omega_T(e_0) = \omega_T(e_1)$ and the cuts generated by
$e_0$ and $e_1$ have an edge in common).
\item
If $\omega_T(e_2) > 2 \min{ \{\omega_T(e_0), \omega_T(e_1) \} }$, then the greedy algorithm
will delete one of $e_0$ and $e_1$ followed by the other, possibly giving a non-optimum result, as
in the graph on the left in Figure \ref{f:greedyfailson2fs}.
\item If $\omega_T(e_2) = 2 \min{ \{\omega_T(e_0), \omega_T(e_1) \} }$ then the greedy algorithm will randomly choose to add either $e_2$ or the minimum of $e_0$ and $e_1$ followed by the other. 
As with cases 2 and 3, this need not be optimal.

\end{enumerate}

\end{itemize}

\subsection{Empirical performance of the Gomory-Hu-based algorithm} \label{s:ghuxpts}

In our experiments, we made the following choices as to the number, size, and construction of forbidden sets.
\begin{enumerate}
\item The graph $G = (V, E)$ was the giant component of a random graph
sampled uniformly from those with degree sequence $(D_1, \ldots, D_n)$, where
$(D_1 - 1, \ldots, D_n - 1)$ was a Multinomial$(2m - n, (1/n, \ldots, 1/n))$ random vector.
Typically $|V| \approx n$, and $|E| \approx m$. 

\item The number of forbidden sets was $b = \lceil\log(n)\rceil$.

\item We selected $a = \lceil 0.75 b\rceil$ ``bad'' vertices at random. The other vertices were ``good'' vertices. 

\item Each forbidden set was constructed by selecting one bad vertex uniformly at random, and two good vertices uniformly at random.
The forbidden sets were statistically independent. Since $b > a$, 
there were cases where the same bad vertex occurred in multiple forbidden sets.

\end{enumerate}

We tried the the following two algorithms on several examples:
\begin{enumerate}
\item {\bf A greedy algorithm based on Gomory Hu trees} as described in Section \ref{s:greedygomhu}.

\item {\bf An exhaustive search algorithm among all proper 2-colorings} along the lines of Proposition \ref{p:onef}.
Given a combinatorial data fusion problem $(V, E, w, \mathcal{F})$, compute the union $X$ of all the forbidden sets,
and compute all proper\footnote{
A 2-coloring is proper if and only if no $F\in\mathcal{F}$ is monochromatic.
} 2-colorings of $(X, \mathcal{F})$. 
For each proper 2-coloring, solve the minimum cut problem in $(V \cup \{s, t\}, E\cup E_0 \cup E_1)$,
where $E_0$ consists of edges between a source $s$ and every vertex in $X$ colored 0, and
$E_1$ consists of edges between a sink $t$ and every vertex in $X$ colored 1.
Output a 2-coloring of X for which the weight of this minimum cut is least.
Note that this algorithm delivers an optimum if and only if there exists an optimum solution in which there are just two colors.
Indeed if there are no feasible 2-colorings then the algorithm will fail.
\end{enumerate}

\begin{table}
\begin{center}
\begin{tabular}{|c|c|c|c|c|c|c|}
\hline
 & Graph Size                    & $|\mathcal{F}|$    & Cut Weight     & Time (secs)  		& Cut Weight   & Time (secs)\\
 & $({\rm Nodes}, {\rm Edges})$  &     	   & (2-color opt.) & (2-color opt.)    & (Greedy)     & (Greedy) \\
\hline
1 & $(60,90)$                    & 5     	   & 2.01           & 0.2     		& 2.01         & 0.02\\
\hline
2 & $(64, 192)$                  & 5    	   & 14.93          & 1.2     		& 14.93        & 0.02\\
\hline
3 & $(1024, 1536)$               & 7    	   & 4.15           & 562      		& 4.51         & 0.4\\
\hline
4 & $(1K, 3K)$                  & 7  		   & 11.72     	    & 1900    		& 11.83        & 0.4\\
\hline
6 & $(32K, 98K)$                 & 11 		   & \mbox{n/a}     & \mbox{n/a}       	& 18.47        &  6948\\
\hline
\end{tabular}
\caption{{\bf 2-color optimum vs. Gomory-Hu Based Greedy Algorithm:} Code run in Java 1.8 on a single core,
on top of the \texttt{JGraphT} implementation of Gomory-Hu Tree.} \label{comparetab}
\end{center}
\end{table}

Observe in Table \ref{comparetab} that the cut weight obtained by both algorithms is close in all examples, and identical in examples 1 and 2.
In all examples the greedy algorithm delivers a solution with two colors, but in problems 3 and 4
it fails to find an optimum coloring.

The compute time for exhaustive 2-coloring reflects exponential growth in the number of proper 2-colorings
of the union of the forbidden sets, and a quadratic growth in the time to solve a single minimum cut problem on the whole of $(V, E, w)$.
We believe that the compute time for the greedy algorithm is dominated by the time to construct a Gomory-Hu tree.

\section{The multiway cut problem}

\subsection{Review} \label{s;muliwaycut}

Return to the multiway cut case, mentioned in Section \ref{s:mmwcut},
 where the forbidden hypergraph $\mathcal{F}$ consists
of the complete graph on a proper subset $T \subset V$.
Assume $|T| \geq 3$, since cases where $|T| < 2$ are vacuous, and the case where $|T|= 2$ 
 is the standard Minimum Cut problem on graphs.
 No edge in $E_0$ can have both endpoints in $T$, but, since $G$ is connected, 
every vertex in $T$ is incident to at least one edge in $E_0$.
A hypergraph coloring for $(V, \mathcal{F})$, in the sense of Section \ref{s:fhg}, means a 
vertex partition $U_1, U_2, \ldots \subset V$ such that $|U_i \cap T| \leq 1$ for all $i$. 
We seek such a partition for which the total weight of edges with endpoints
in different $(U_i)$ is minimum.
By Lemma \ref{l:equiv}, this is the same as a maximum
weight edge subset $E_1 \subset E_0$ such that no component of the graph $(V, E_1)$
contains more than one element of $T$; we seek to remove
edges of minimum total weight, such that
each of the vertices in $T$ must be in a different graph component after the cut.

\subsubsection{Application: vertex label anomaly correction} \label{s:anomaly}
Darling and Velednitsky \cite{dar} consider a large bipartite graph encoding transactional data. In an abstract sense, the left vertex set $L$
consists of clients, the right vertex set $R$ consist of servers, and each edge records the transaction
volume between a specific client - server pair.
Furthermore there is a set $C$ of categories, and each right vertex is connected by a weighted edge to exactly one of these categories.
The complete vertex set is $V = L \cup R \cup C $, and the edge set $E_0$ includes both the transaction records and the category labelling
of right vertices.
The ideal state would be that each component of this graph contains exactly one vertex of $C$, which leads to a partition of $L \cup R$
among the categories. If this is not the case, we seek a maximum weight subset $E_1 \subset E_0$ for which this condition holds.
This is a multiway cut problem: all vertices in $C$ must belong to different components.

\subsection{Matroid perspective}

Given $T \subset V$ with $2 \leq |T| < |V|$, and given a connected weighted graph
$G=(V, E_0, w)$ such that no edge in $E_0$ has both endpoints in $T$,
define $\mathcal{E}_T$ to be the collection of subsets $X \subseteq E_0$ such that:
\begin{enumerate}
\item
$(V, X)$ is a forest, and
\item
None of the tree components of the graph $(V, X)$
has more than one vertex in $T$.

\end{enumerate}

\begin{figure}
\caption{\textbf{Construction for matroid proof in Proposition \ref{p:mwcmatroid} }
\textit{The set $T=\{t_1, t_2, t_3\}$. The auxiliary vertex is $t_0$.
}
} \label{f:mwcutcyclematroid}

\begin{center}
\scalebox{0.25}{\includegraphics{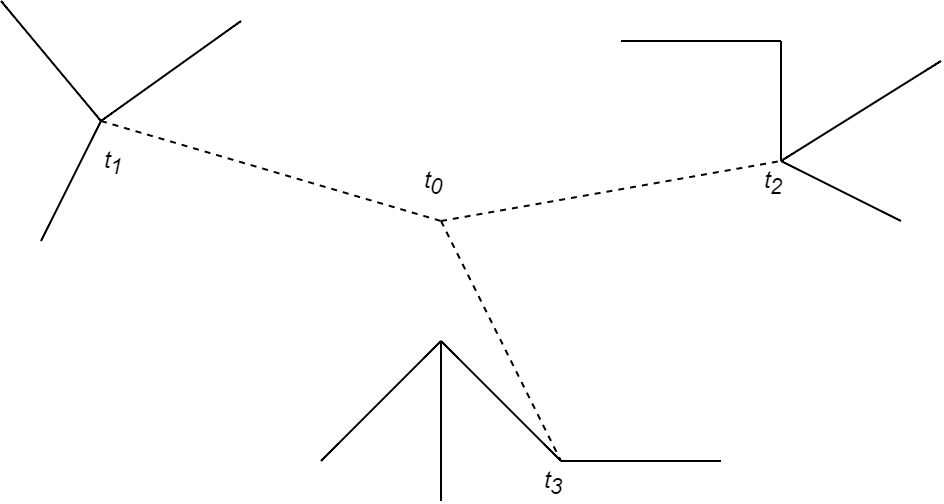}} 
\end{center}

\end{figure}

We call $\mathcal{E}_T$ the \textbf{multiway cut cycle matroid}, because of
Proposition \ref{p:mwcmatroid}.

\begin{proposition} \label{p:mwcmatroid}
$\mathcal{E}_T$ is a matroid.
\end{proposition} 

Before proving Proposition \ref{p:mwcmatroid}, we introduce an auxiliary construction,
illustrated in Figure \ref{f:mwcutcyclematroid}.
Let $T=\{t_1, t_2, \ldots, t_k\}$. Introduce a new graph $G':=(V', E')$, on the vertex set
$V':= V \cup \{t_0\}$, where $t_0$ is a new vertex not in $V$,
and with edge set
\[
E':= E_0 \cup E_*; \quad E_*:=\{ \{t_0, t_i\}_{1 \leq i \leq k} \} .
\]
In other words, $E'$ is obtained by augmenting $E_0$ with a star centered at the new vertex $t_0$,
connected to each of the vertices in $T$; see Figure \ref{f:mwcutcyclematroid}.

\begin{lemma}
There is a bijection between $\mathcal{E}_T$ and the edge sets $Z \supseteq E_*$, for which
$(V', Z)$ is a forest contained in $(V', E')$.
\end{lemma}

\begin{proof}
Given $X \in \mathcal{E}_T$, define $X_i \subset X$
to be those edges in the component containing $t_i$, for $1 \leq i \leq k$.
In other words, for $1 \leq i \leq k$, $X_i \subset X$ is maximal such that
 $(V(X_i), X_i)$ is connected and $V(X_i) \ni t_i$.
Let $X_0$ denote the edges in all components with no vertices in $T$. Hence
\[
X = X_0 \cup X_1 \cup \cdots \cup X_k
\]
is a partition of $X$. By definition of $\mathcal{E}_T$,  none of the graphs
$(V(X_i), X_i)_{0 \leq i \leq k}$ contain a cycle. The graph $(V', E_*  \cup X)$
contains a cycle-free component with edge set $E_*  \cup (X \setminus X_0)$,
whose vertex set is
\[
\{t_0, t_1, t_2, \ldots, t_k\} \cup \bigcup_{i=1}^k V(X_i).
\]
In this component, the unique path between $t_i$ and $t_j$ for $1 \leq i < j \leq k$ passes through $t_0$.
The forest $(V, X_0)$ is unaffected by adding edges in $E_*$, and remains cycle-free. Hence 
$(V', Z:=E_*  \cup X)$ is a forest as claimed. The same argument can be reversed to give
the desired bijection.
\end{proof}

\subsection{Proof of Proposition \ref{p:mwcmatroid}}
\begin{proof}
\textbf{Step I: cycle matroid: }
For the sake of completeness, we recapitulate from \cite{kor}, Proposition 13.4, 
the short proof that for any graph $(V, E)$,
the collection $\mathcal{T}$, consisting of subsets $X\subseteq E$ for which $(V, X)$ is a forest,
 is a matroid. It is clear $\mathcal{T}$ is an independence system.
Suppose $X$ and $Y$ both belong to $\mathcal{T}$, but $X \cup \{e \} \notin \mathcal{T}$
for all $e \in Y$. For every edge $e$ in $Y$, both endpoints of $e$ are in the same component
of $(V, X)$, so each connected component of $(V, Y)$ is a subset of a connected component of $(V, X)$.
Hence $|C(X)| \leq |C(Y)|$, where $C(X)$ denotes the connected components of $(V, X)$.
It follows that
\[
|X| = |V| - |C(X)| \geq |V| - |C(Y)| = |Y|.
\]
So if $X, Y \in \mathcal{T}$ and $|X| < |Y|$, then $X \cup \{e \} \in \mathcal{T}$
for some $e \in Y$, proving that $\mathcal{T}$ is a matroid. It is called the \textbf{cycle matroid}
of $(V, E)$

\textbf{Step II: matroid contraction: }
To conclude, we show that $(E_0, \mathcal{E}_T)$ is the contraction of another matroid, and therefore a matroid.
First recall some terminology. 
If $M = (E,I)$ is a matroid and $X \subset E$ then the \textbf{deletion matroid} $M \setminus X = ( E \setminus X, I')$ 
where $I'  = \{ Z | Z \subset (E \setminus X), Z \in I\}$.
The \textbf{dual matroid} $M^* = (E, I^*)$, where $I^*$ consists of those sets $F \subset E$ for which
there is a base $B$ of $M$ such that $F \cap B = \emptyset$. For $X \subset E$, the 
\textbf{matroid contraction} $M / X = (M^* \setminus X)^*$.
Note that if $X \subset E$, and $Z$ is a basis of $X$, then a subset $B$ of $E \setminus X$ is independent in 
$M / X$ if and only if $B \cup Z$ is independent in $M=(E,I)$.  

In this setting, take $M = (E',I)$ to be the cycle matroid of $(V', E_* \cup E_0)$ .  
Since $E_*$ is independent in $M$ it is a base for itself. 
Therefore the independent sets of the $M / E_*$ are exactly the forests of $G$ with 
the $t_i$ in different components (A forest $F$ in $G$ with $t_0$ and 
the edges of $E_*$ added is acyclic if and only if the $t_i$ are in different components in $F$.)
Hence the matroid $M / E_*$ is $(E_0, \mathcal{E}_T)$ as desired.
\end{proof}

\subsection{Greedy multiway cut forest algorithm}

As for any matroid, the \textbf{Best-In Greedy Algorithm} takes the following form.
\begin{enumerate}
\item
Sort $E_0 = \{e_1, e_2, \ldots, e_n\}$  such that $w(e_1) \geq w(e_2) \geq \cdots \geq w(e_n)$.

\item
Set $X = \emptyset$.

\item
For $i = 1$ to $n$, \textbf{do: } if $X \cup \{e_i\} \in \mathcal{E}_T$, then set $X:= X \cup \{e_i\}$.

\end{enumerate}

The context here is multiway cut with respect to $T \subset V$ with $2 \leq |T| < |V|$.
The connected weighted graph $G=(V, E_0, w)$ has no edge in $E_0$ with both endpoints in $T$.
In practice, Step 3. entails checking that the next edge $e_i$ neither causes a cycle, nor causes
two elements of $T$ to fall in the same component.

Proposition \ref{p:mwcmatroid} has the Corollary:

\begin{corollary} 
Suppose the forbidden hypergraph consists
of the complete graph on a proper subset $T \subset V$ with $2 \leq |T| < |V|$, and 
the connected weighted graph $G=(V, E_0, w)$ has no edge in $E_0$ with both endpoints in $T$.
The Best-In Greedy Algorithm finds an optimum solution to the problem of
finding a maximum weight forest $E_1 \subset E_0$ such that no component of $(V, E_1)$ contains
more than one element of $T$.
\end{corollary} 

Call this $E_1$ an \textbf{optimum multiway cut forest}.

\begin{proof}
Since $\mathcal{E}_T$ is a matroid by Proposition \ref{p:mwcmatroid},
the Corollary follows from the Edmonds-Rado Theorem, \cite[Theorem 13.20]{kor}.
\end{proof}

\textbf{Remark: }
Costa \& Billionnet \cite{cos} have given a construction which solves 
the multiway cut problem on trees in linear time. 
The Best-In Greedy Algorithm falls short of linear time in that (1)
the sort alone is $O(|V| \log (|V|))$. and (2)
for the overall runtime to be linear in $|V|$, 
the check for each proposed edge would need to be done in constant time.

\textbf{Acknowledgments: } The authors thank Imtiaz Manji and Adam Logan of the Tutte Institute of
 Mathematics and Computing for insightful comments, and the organizers and participants
of STAMP 2017 for encouraging this work.

%%%%%%%%%%%%%%%%%%%%%%%%%%%%%%%%%%%%%%%%%%%%%%%%%%%%%%%%%%%%%%%%%%%%%%%%%%%%%%%%%%%%%%%%%%%%

\end{document}